\documentclass[12pt]{amsart}
\usepackage{txfonts}
\usepackage{amsfonts}
\usepackage{amsfonts}
\usepackage{amsmath}
\usepackage{amsfonts}
\usepackage{amssymb,latexsym}
\usepackage{cite}
\usepackage{enumerate}
\usepackage{ulem}
\newtheorem{theorem}{Theorem}

\newtheorem{remark}[theorem]{Remark}
\newtheorem{proposition}[theorem]{Proposition}
\newtheorem{lemma}[theorem]{Lemma}

\newcommand {\p}{\partial}

\numberwithin{equation}{section}
\numberwithin{theorem}{section}

\usepackage{titletoc}

\usepackage{hyperref}
\hypersetup{hypertex=true,colorlinks=true,linkcolor=blue,anchorcolor=blue,citecolor=blue}
\begin{document}
\title[The homogeneous k-Hessian equation]{The  Dirichlet problem of the homogeneous $k$-Hessian equation in a punctured domain}
\author{Zhenghuan Gao}
\address{Department of Mathematics, Shanghai University, Shanghai, 200444, China}
\email{gzh@shu.edu.cn}
\author{Xi-Nan Ma} 
\address{School of Mathematical Sciences, University of Science and Technology of China, Hefei 230026,
	Anhui Province, China}
\email{xinan@ustc.edu.cn}

\author{Dekai Zhang}  
\address{Department of Mathematics, Shanghai University, Shanghai, 200444, China}
\email{dkzhang@shu.edu.cn}

\begin{abstract}
In this paper, we consider the  Dirichlet problem for the homogeneous $k$-Hessian equation with prescribed asymptotic behavior at $0\in\Omega$ where  $\Omega$ is a $(k-1)$-convex bounded domain in the Euclidean space.  The prescribed asymptotic behavior at $0$ of the solution is zero if $k>\frac{n}{2}$, it is $\log|x|+O(1)$ if $k=\frac{n}{2}$ and   $-|x|^{\frac{2k-n}{n}}+O(1)$ if $k<\frac{n}{2}$. 
To solve this problem, we consider the Dirichlet problem of the approximating $k$-Hessian equation in $\Omega\setminus \overline{B_r(0)}$ with $r$ small. 
We firstly construct the subsolution of the approximating $k$-Hessian equation. Then we  derive the pointwise $C^{2}$-estimates of the approximating equation based on  new  gradient and second order estimates established previously by the second author and the third author. In addition, we  prove a uniform positive lower bound of the gradient if the domain is  starshaped with respect to $0$. As an application, we prove an identity along the level set of the approximating solution and obtain a nearly monotonicity formula. In particular, we get a weighted geometric inequality for smoothly and strictly $(k-1)$-convex  starshaped closed hypersurface in $\mathbb R^n$ with $\frac{n}{2}\le k<n$.

\end{abstract}
\maketitle.
\tableofcontents 

\section{Introduction}
Let  $\Omega$ be a bounded domain in $\mathbb{R}^n$ and $u\in C^2(\Omega)$. The $k$-Hessian operator $F_k[u]$ is defined by 
\begin{align}
	F_k[u]:=S_k(D^2 u),
\end{align}
where $S_k(D^2 u)$ is the sum of all 
principal $k\times k$ minors of $D^2 u$. If $\lambda=(\lambda_1,\cdots,\lambda_n)$ are the eigenvalues of $D^2u$, one can see that $S_k(D^2 u)=\sum_{1\le i_1<\cdots<i_k\le n} \lambda_{i_1}\cdots\lambda_{i_k}$.

Caffarelli-Nirenberg-Spruck \cite{CNSIII} solved  the following Dirichlet problem for the $k$-Hessian equation
\begin{align}\label{khessian}
	\left\{\begin{aligned}
		S_k(D^2 u) =&f \qquad\text{in} \quad \Omega,\\
		u=&\varphi \qquad\text{on} \quad \partial\Omega,
	\end{aligned}
	\right.
\end{align}
where $f>0$ and $\varphi$ are given smooth functions.
By assuming the existence of a subsolution, Guan \cite{guan1994cpde, guan2014duke} solved \eqref{khessian}.

For the degenerate case i.e. $f\ge 0$,
Wang \cite{Wang1994} solved  the Dirichlet problem: $S_k(D^2u)=f(x,u)$ in $\Omega$, $u=0$ on $\partial{\Omega}$ and proved the Sobolev-type inequality for the related functional $\int_{\Omega}u S_{k}(D^2 u) dx$.

Wang-Chou \cite{cw2001cpam} used the parabolic method to prove the existence of $k$-convex solutions $u$ to the problem  $S_k(D^2 u)=f(x,u)$ in $\Omega$, $u=0$ on $\partial{\Omega}$, where $\Omega$ is strictly $(k-1)$-convex. In \cite{cw2001cpam}, Wang-Chou established the important Pogorelov type second order estimate for the $k$-Hessian equation.

 Krylov\cite{krylov1989, krylov1994}	proved the $C^{1,1}$ regularity of the  problem: $S_k(D^2 u)=f(x)$ in $\Omega$ and $u=\varphi$ on $\partial \Omega$ by assuming $f^{\frac{1}{k}}\in C^{1,1}$ , $\varphi\in C^2$ and $(k-1)$-convexity of $\Omega$. Ivochina-Trudinger-Wang\cite{itw2004cpde} gave a new and simple proof.
Li-Luc\cite{LiLuc} studied the existence and uniqueness of the Green's function for the nonlinear Yamabe equation.

In the seminal papers\cite{trudingerwang1997tmna,tw1999am,trudingerwang2002jfa}, Trudinger-Wang studied systematically the Hessian measure for the $k$-convex function in $\mathbb{R}^n$ where they only assume that the function  was continuous, locally bounded and locally integrable respectively. Labutin \cite{labutin2002duke} continued to  study the potential theory of the $k$-Hessian measure. 

The  fundamental solutions of the $k$-Hessian equation are as follows
\begin{align}
	G_{k}(x)=\left\{
		\begin{aligned}
			-|x|^{2-\frac{n}{k}}&\quad\text{if}\quad k<\frac{n}{2},\\
			\log|x|&\quad\text{if}\quad k=\frac{n}{2},
			\\
			|x|^{2-\frac{n}{k}}&\quad\text{if}\quad k>\frac{n}{2}.
			\end{aligned}
		\right.
	\end{align}

In this paper, we want to study the regularity problem for the homogeneous $k$-Hessian equation in $\Omega\setminus{0}$.

In the complex Euclidean space, 
 Klimek \cite{Klimek1985} introduced the extremal fucntion
$$g_\Omega(z,z_0)=\sup\{v\in \mathcal{PSH}(\Omega):v<0,\ v(z)\leq \log|z-z_0|+O(1)\}.$$
$g_\Omega(z,\xi)$ is  called the pluricomplex Green function on $\Omega\subset\mathbb{C}^n$ with a logarithminc pole at $z_0$. If $\Omega$ is hyperconvex, Demailly \cite{Demailly1987} showed that $u(z)=g_\Omega(z,z_0)$ is continuous and solves uniquely the following homogeneous complex Monge-Amp\`ere equation
\begin{align}\label{eqMA}
	\begin{cases}
		(dd^cu)^n=0&\quad\text{in }\Omega\setminus\{z_0\},\\
		u=0&\quad\text{on }\partial\Omega,\\
		u(z)=\log|z-z_0|+O(1)&\quad\text{as }z\rightarrow z_0.
	\end{cases}
\end{align}
If $\Omega$ is strictly convex  with smooth boundary, Lempert \cite{Lempert1981} proved the solution is smooth. For the strongly pseudonconvex case, B. Guan \cite{GuanBo1998} proved $ C^{1,\alpha}$ regularity and later, B\l ocki \cite{Blocki2000}  showed the  $C^{1,1}$ regularity. The $C^{1,1}$ regularity is optimal by the counterexamples  by Bedford-Demailly \cite{BedfordDemailly1988}, . 

P. Guan \cite{gpf2002am} established the $C^{1,1}$ regularity of the extremal function associated to intrinsic norms of Chen-Levine-Nirenberg \cite{CLN1969} and Beford-Taylor \cite{BedfordTaylor1979} where the extremal function solves
\begin{align*}
	\begin{cases}
		(dd^cu)^n=0&\quad\text{in }\Omega_0\setminus(\cup_{i=1}^m \Omega_i),\\
		u=0&\quad\text{on }\partial\Omega_i, \ i=1,\cdots,n\\
		u=1&\quad\text{on }\partial\Omega_0.
	\end{cases}
\end{align*}

\subsection{Our main results}
Motivated by Labutin's work \cite{labutin2002duke} and Guan's work \cite{GuanBo1998}, we consider the following  Dirichlet problem for the homogeneous $k$-Hessian equation with interior isolated singularities. For convenience,  we assume the singularity is  $0\in \Omega$ and  there exists positive constants $r_0, {R_0}$ such that $B_{r_0}\subset\subset\Omega\subset\subset B_{ {R_0} }$, where $B_r$ and $B_{ {R_0}}$ are balls centered at $0$ with radius $r$ and ${R_0}$ respectively. 

We divide three cases to state our main results.

\subsubsection{\emph{\textbf{Case1:}} $k>\frac{n}{2}$}
In this case, since the fundamental solution of the homogeneous $k$-Hessian equation is $|x|^{2-\frac{n}{k}}$ which tends to $0$ as $x\rightarrow 0$, we consider the following problem
\begin{equation}\label{case2Equa1.2}
	\left\{
	\begin{aligned} S_{k}(D^2 u)=0 \qquad\ \ \ &\text{in}\quad \Omega\setminus\{0\},\\
		u=1\qquad \qquad \
		\  \  &\text{on}\ \ \  \partial\Omega,\\
		\lim\limits_{|x|\rightarrow 0}u(x)=0.\quad\qquad&
	\end{aligned}
	\right.
\end{equation}
We prove the following uniqueness and existence result.
\begin{theorem}\label{main07201}
	Assume $ k>\frac{n}{2}$. Let $\Omega$ be a smoothly convex domain in $\mathbb{R}^n$ and strictly $(k-1)$-convex. There exists a unique  $k$-convex solution $u\in C^{1,1}(\overline{\Omega^c})$ of the equation \eqref{case2Equa1.2}. Moreover, there exists uniform constant $C$ such that for any $x\in \Omega^c$ the following holds
	\begin{align}\label{decay20720}
		\left\{
		\begin{aligned}
			C^{-1}|x|^{\frac{2k-n}{k}}\le u(x)\le& C|x|^{\frac{2k-n}{k}},\\
			 |Du|(x)\le& C|x|^{\frac{k-n}{k}},\\
			|D^2u|(x)\le& C|x|^{-\frac{n}{k}}.
		\end{aligned}
		\right.
	\end{align}
\end{theorem}
\subsubsection{\emph{\textbf{Case2:}} $1\le k<\frac{n}{2}$}
We consider the following problem
\begin{equation}\label{case1Equa1.1}
\left\{\begin{aligned} S_{k}(D^2 u)=0 \qquad&\text{in}\quad \Omega\setminus\{0\},\\
u=-1\quad &\text{on}\ \partial\Omega,\\
u(x)=-|x|^{2-\frac{n}{k}}+O(1)\quad &\text{as}\ x\rightarrow 0.\end{aligned}
\right.
\end{equation}
If we  prescribe  $u=-C_0|x|^{2-\frac{n}{k}}+O(1)$ as $x\rightarrow 0$ for some positive constant $C_0$, then $\tilde u=C_0^{-1} u+C_0^{-1}-1$ solves \eqref{case1Equa1.1}. 
\begin{theorem}\label{main07202}
Assume $1\le k<\frac{n}{2}$. Let $\Omega$ be a smoothly, strictly $(k-1)$-convex domain in $\mathbb{R}^n$. There exists a unique  $k$-convex solution $u\in C^{1,1}(\overline{\Omega}\setminus\{0\})$ of the equation \eqref{case1Equa1.1}. Moreover, there exists uniform constant $C$ such that for any $x\in \overline\Omega\setminus\{0\}$,  the following holds
\begin{align}\label{decay10720}
\left\{
\begin{aligned}
 \Big|u(x)-|x|^{-\frac{n-2k}{k}}\Big|\le& C,\\
 |Du|(x)\le& C|x|^{-\frac{n-k}{k}},\\
|D^2u|(x)\le& C|x|^{-\frac{n}{k}}.
\end{aligned}
\right.
\end{align}
\end{theorem}
\begin{remark}
Assume $1\le k\le \frac{n}{2}$. Labutin \cite{labutin2002duke} proved if $u$ is $k$-convex solving $S_k(D^2 u)=0$ in $B_R\setminus \{0\}$, $u<0$   and $0$ is the singular point of $u$, there exists a positive constant $C_0$ such that $u(x)=C_0G_k(x)+O(1)$ as $x\rightarrow 0$. This is the reason why we  prescribe  the above asymptotic behavior in \eqref{case1Equa1.1}.

\end{remark}

\subsubsection{\emph{\textbf{Case3:}} $k=\frac{n}{2}$}
Since the Green function in this case is $\log|x|$, we  consider the $k$-Hessian equation when $k=\frac{n}{2}$ as follows
\begin{equation}\label{case3Equa1.3}
\left\{\begin{aligned} S_{\frac{n}{2}}(D^2 u)=0 \ \ \qquad&\text{in}\ \ \Omega\setminus\{0\},\\
u=0\ \ \qquad &\text{on}\ \partial\Omega,\\
u(x)=\log|x|+O(1) \ &\text{as}\ {|x|\rightarrow 0}.
\end{aligned}
\right.
\end{equation}
If we prescribe  $u=C_0\log|x|+O(1)$ as $x\rightarrow 0$ for some positive constant $C_0$, then  $\tilde u=C_0^{-1} u$  solves \eqref{case3Equa1.3}. 
\begin{theorem}\label{main07203}
Assume $k=\frac{n}{2}$. Let $\Omega$ be a smoothly and strictly $(k-1)$-convex domain in $\mathbb{R}^n$. There exists a unique  $k$-convex solution $u\in C^{1,1}(\overline{\Omega}\setminus\{0\})$ of the equation \eqref{case3Equa1.3}. Moreover, there exists uniform constant $C$ such that for any $x\in \overline\Omega\setminus\{0\}$ the following holds
\begin{align}\label{decay30720}
\left\{
\begin{aligned}
 |u(x)-\log|x||\le& C,\\
  |Du|(x)\le& C|x|^{-1},\\
|D^2u|(x)\le& C|x|^{-2}.
\end{aligned}
\right.
\end{align}
\end{theorem}

To solve the above problems, for example when $k>\frac{n}{2}$ we will prove there exists a smooth $k$-convex function $u^{\varepsilon}$ solving
\begin{align*}
	\left\{
	\begin{aligned}
		 S_{k}(u^{\varepsilon})=\varepsilon &\quad\text{in}\quad \Omega\setminus\{0\},\\
		 u^{\varepsilon}=1 &\quad\text{on} \quad \partial \Omega, \\
		\lim\limits_{|x|\rightarrow 0}u^{\varepsilon}(x)\rightarrow0.&
		\end{aligned}
	\right.
	\end{align*}
Note that the right hand side of the above approximating equation  is $\varepsilon$ which is different from  the exterior Dirichlet problem case.
To solve the above approximating  equation, we  consider the approximating $k$-Hessian equation in $\Omega_r:=\Omega\setminus\overline B_r$ and we will prove the uniform  $C^{1,1}$-estimates.  We firstly construct a subsolution of the approximating $k$-Hessian equation in $\Omega_r$.  This follows from a key lemma due to P. F. Guan \cite{gpf2002am} by the $(k-1)$-convexity of the domain. Note that the second and third author have proved the global gradient and second order estimate in \cite{MaZhang2022arxiv}. Thus we only need to prove the boundary estimates.

%
\subsection{Applications to the starshaped $(k-1)$ convex domain}
As an application of our $C^{2}$ estimates for the approximating equation, we can prove an almost monotonicity formula along the level set of $u^{\varepsilon}$ when $\Omega$ is additionally starshaped. Consequently, we get some weighted geometric inequalities of $\p\Omega$ when $\frac{n}{2}\le k< n$.
\begin{theorem}\label{geometric0725}
	Let $\Omega$ be a bounded smooth starshaped  domain with respect to $0$ in $\mathbb{R}^n$ and strictly $(k-1)$-convex.
	\begin{enumerate}[{(i)}]
		\item
	Assume $\frac{n}{2}<k< n$. Assume $b\ge\frac{k(n-k-1)}{n-k}$. Let $u$ be the unique $C^{1,1}$ solution in Theorem \ref{main07201}. We have
		\end{enumerate}
	\begin{align}\label{0727geometric}
	\int_{\p\Omega}{|Du|^{b+1}H_{k-1}}	\ge \frac{2k-n}{n-k}\int_{\p\Omega}{|Du|^{b}H_{k}},
	\end{align}
	\end{theorem}
where $H_m$ is the $m$-Hessian operator of the principal curvature $\kappa=(\kappa_1,\cdots,\kappa_{n-1})$ of $\p\Omega$.
	\begin{enumerate}[{(ii)}]
		\item
Assume $k=\frac{n}{2}$ and $b\ge\frac{n}{2}-1$. Let $u$ be the unique $C^{1,1}$ solution in Theorem \ref{main07203}. We have
\end{enumerate}
\begin{align}
	\int_{\p\Omega}|Du|^{b+1}H_{k-1}
	\ge \int_{\p\Omega}|Du|^{b}H_{k}.
\end{align}
\begin{remark}
	If we assume $\Omega$ is starshaped with respect to $x_0\in \Omega$, the above inequality still holds for  $u$ which solves the homogeneous $k$-Hessian equation in $\Omega\setminus\{x_0\}$.	\end{remark}

 \textbf{Organization of this paper.} In section 2,  we firstly construct a subsolution for the approximating equation  by a lemma due to P. F. Guan \cite{gpf2002am}. Based on the new gradient and second order estimates in \cite{MaZhang2022arxiv}, we show uniform $C^{1,1}$ estimate of the approximating solution. The positive lower bound of the gradient of the approximating solution is proved if we also assume $\Omega$ is starshaped. 
Theorem \ref{main07201}, Theorem \ref{main07202} and Theorem \ref{main07203} will be proved in Section 4. In section 5, we prove an almost monotonicity formula along the level set of the approximating solution and then we show Theorem \ref{geometric0725}.

\section{Solving the approximating equation in $\Omega_r:=\Omega\setminus B_r$.}
We need the following  lemma by P. F. Guan\cite{gpf2002am} to construct the existence of the subsolution of the $k$-Hessian equation in $\Omega\setminus \overline{B_r}$.
\begin{lemma}\label{Guan2002lemma}
	Suppose that $U$ is a bounded smooth domain in $\mathbb{R}^n$. For $h, g\in C^m(U)$, $m\ge 2$, for all $\delta>0$, there is an $H\in C^m(U)$ such that
	
	\begin{enumerate}[(1)]
		\item
		$H\ge \max\{h, g\}$ \ \text{and}
		
		\begin{align*}
			H(x)=\left\{ {\begin{array}{*{20}c}
					{h(x), \quad   \text{if } \ h(x)-g(x)>\delta  }, \\
					g(x) , \ \quad \text{if } \ g(x)-h(x)>\delta;\\
			\end{array}} \right.
		\end{align*}
		\item
		{There exists}  $|t(x)|\le 1 $ {such that}
		
		\begin{align*}
			\left\{H_{ij}(x)\right\}\ge
			\left\{\frac{1+t(x)}{2}g_{ij}+\frac{1-t(x)}{2}h_{ij}\right\},\ \text{for all} \ x\in\left\{|g-h|<\delta\right\}.
		\end{align*}
	\end{enumerate}
\end{lemma}
{By the convacity of $S^\frac1k$, we} can prove that $H$ is  $k$-convex if $f$ and $g$ are both $k$-convex.

 Recall that we always assume $B_{r_0}\subset\subset\Omega\subset\subset B_{(1-\tau_0)R_0} $ for some  $\tau_0\in (0,\frac{1}{2})$.
 Firstly we state a useful fact for the strictly $(k-1)$-convex domain{, which can be found in \cite[Section 3]{CNSIII}.}
 \begin{lemma}\label{lem2.2}
 	Let $\Omega$ be a smoothly and strictly $(k-1)$-convex bounded domain. 
 	There exists $\mu_0>0$ small such that $\Omega_{2\mu_0}:=\{x\in \Omega: d(x)<2\mu_0\}$ is close to $\p \Omega$ ,$B_{r_0}\subset\subset\{x\in \Omega:d(x)> 2\mu_0\}$ and $d(x)$ is smooth in $\overline\Omega_{2\mu_0}$. Moreover,  $\Phi^0:=t_0^{-1}(e^{-t_0d(x)}-1)$ is smooth and strictly $k$-convex and $S_k(D^2(\Phi^0))\ge \epsilon_0$ in $\overline\Omega_{2\mu_0}$ for some uniform positive constants $t_0$ and $\epsilon_0$. 
 \end{lemma}
 
 \subsection{Case 1: $k>\frac{n}{2}$}
 Since the Green function in this case is $|x|^{\frac{2k-n}{k}}$, we want to solve the $k$-Hessian equation as follows
 \begin{equation}\label{case2Equa}
 	\left\{\begin{aligned}S_{k}(D^2 u)=&0 \quad\text{in}\ \ \  \mathring{\Omega},\\
 		u=&1\ \ \ \ \text{on}\ \partial\Omega,\\
 		\lim\limits_{x\rightarrow 0}u(x)=&0.
 	\end{aligned}
 	\right.
 \end{equation}

 \subsubsection{The approximating equation }
 We will use the solution of a sequence of nondegenetare equations  in $\Omega_r$ to approximate the solution of the homogeneous $k$-Hessian equation. The existence of the approximating solution can be obtained if we can construct a smooth subsolution. We use the $(k-1)$-convexity of $\p\Omega$ and the  Lemma \ref{Guan2002lemma} by P. F. Guan \cite{gpf2002am} to prove the existence of the subsolution. 
 
 Denote $w:=\frac{1}{2}
 {\Big(\frac{|x|}{R_0}\Big)}^{2-\frac{n}{k}}+\frac{|x|^2}{2R_0^2}$. 
 {By the concavity of $S_k^\frac1k$,}
\begin{equation*}
{S_k^\frac1k(D^2w)=S_k^\frac1k\Big(\frac12D^2\big(\frac{|x|}{R_0}\big)^{2-\frac{n}2}+\frac{1}{2R_0^2}D^2|x|^2\Big)\geq S_k^\frac1k(\frac{1}{R_0^2}I).}
\end{equation*}
Then we have
 \begin{align*}
 	S_{k}(D^2w)
 	\ge { C_n^k}R_0^{-2k}.
 \end{align*}
 Then we construct a smoothly and strictly $k$-convex function $\underline {u}$ by lemma \eqref{Guan2002lemma} as follows. 
 \begin{lemma}
 	There exists a  strictly $k$-convex function $\underline {u}\in C^{\infty}{(\overline\Omega_r)}$ satisfying
 	\begin{align}
 		\underline{u}=&
 		\left\{\begin{aligned}\label{0817:sub1}
 			K_0\Phi^0+1
 			\quad&\text{if}\ d(x)\le \frac{\mu_0}{M_0}, \\
 			w\quad &\text{if}\ d(x)>\mu_0,
 		\end{aligned}
 		\right.\\
 		\underline{u}\ge& \max\left\{ w, K_0\Phi^0+1\right\}\ \ \text{if}\ \frac{\mu_0}{M_0}\le d(x)\le \mu_0,\notag\\
 		S_k(D^2 \underline{u})\ge&\epsilon_1:= \min\{{C_n^k}R_0^{-2k}, {K_0^k}\epsilon_0\} \quad \text{in}\ \Omega,\notag
 	\end{align}
 	where $K_0=\frac{2t_0}{1-e^{-\mu_0t_0}}$ and $M_0$ is determined by $K_0(1-e^{-t_0\frac{\mu_0}{M_0}})=t_0\delta$ with $\delta:=\frac{1}{2}(1-\tau_0)^{2-\frac{n}{k}}$.
 \end{lemma}
 \begin{remark}
 	This lemma tells us that $\underline{u}$ is $K_0\Phi^0+1$ near $\p\Omega$ and $\underline{u}$ is $w$ outside $\Omega_{2\mu_0}$. Moreover, $\underline{u}$ is smooth and  strictly $k$-convex. Although  this lemma is elementary, it is crucial for the proof of $C^{1,1}$ estimates.
 \end{remark}
 \begin{proof}
 	Applying Guan's lemma for $U=\Omega_{2\mu_0}:=\{x\in\Omega: d(x)<2\mu_0\}$, $g=K_0\Phi^0+1$, $h=w$ and $\delta=\frac{1}{2}(1-\tau_0)^{2-\frac{n}{k}}$, we get a strictly and smoothly $k$-convex function $\underline{u}$ in $\Omega_{2\mu_0}$. In the following, we prove \eqref{0817:sub1}.
 	
 	For any $x\in \overline\Omega_{2\mu_0}\setminus\Omega_{\mu_0}:=\{x\in\overline\Omega:\mu_0\le d(x)\le {2\mu_0}\}$, since  $K_0=\frac{2t_0}{1-e^{-t_0\mu_0}}$, we have $$g(x)\le -1.$$
 	Then
 	\begin{align}
 		h-g\ge -g\ge 1>\delta\quad\text{in}\ \overline\Omega_{2\mu_0}\setminus\Omega_{\mu_0}.
 	\end{align}
 	This implies $\underline{u}=w$ in $\overline\Omega_{2\mu_0}\setminus\Omega_{\mu_0}$.
 	
 	For any $x\in\overline{\Omega_{\frac{\mu_0}{M_0}}}:=\{x\in\overline\Omega:d(x)\le \frac{\mu_0}{M}\}$, since $\Omega\subset\subset B_{(1-\tau_0)R_0}$, we have
 	\begin{align}
 		g-h=& t_0^{-1}K_0(e^{-t_0d(x)}-1)+1-\frac{1}{2}
 		{\Big(\frac{|x|}{R_0}\Big)}^{\frac{2k-n}{k}}-\frac{|x|^2}{2R_0^2} \notag\\
 		\ge& t_0^{-1}K_0(e^{-t_0\frac{\mu_0}{M}}-1)+1-(1-\tau_0)^{2-\frac{n}{k}}\notag\\
 		\ge& \frac{1}{2}(1-(1-\tau_0)^{2-\frac n k})=\delta,
 	\end{align}
 	where $M_0$ is defined by $K_0(1-e^{-t_0\frac{\mu_0}{M_0}})=t_0\delta$.
 	This implies $\underline{u}=K_0\Phi^0+1$ in $\Omega_{\frac{\mu_0}{M_0}}$.
 	
 	At last, we define $\underline{u}=w$ in $\Omega_r\setminus\Omega_{2\mu_0}$.  {In $\Omega_\frac{\mu_0}{M_0}$, by Lemma \ref{lem2.2}, $S_k(D^2\underline u)=S_k(K_0\Phi^0)\geq K_0^k\epsilon_0$. In $\Omega_r\backslash \Omega_{2\mu_0}$, $S_k(D^2\underline u)=S_k(D^2w)\geq C_n^kR_0^{-2k}$. In $\Omega_{2\mu_0}\backslash \Omega_{\frac{\mu_0}{M_0}}$, by the concavity of $S_k^\frac1k$,  $S_k^\frac1k(D^2\underline u)\geq \frac{1+t(x)}{2}S_k^\frac1k(D^2w)+\frac{1-t(x)}2S_k^\frac1k(K_0D^2\Phi^0).$ }
 	The proof is complete.
 \end{proof}
 
 Now we consider the following approximating equation 
 \begin{equation}\label{case1EquaAppr}
 	\left\{ \begin{aligned}
 		S_k(D^2 u)=\varepsilon  \quad& \text{in}\    \Omega\setminus\overline {B_r},\\
 		u=1\quad  &\text{on} \  \partial\Omega, \\
 		u=\underline{u}=\frac{1}{2}
 		{\Big(\frac{r}{R_0}\Big)}^{2-\frac{n}{k}}+\frac{r^2}{2R_0^2}\quad &\text{on}\   \partial B_r .  \\
 	\end{aligned} \right.
 \end{equation}
 If $\varepsilon<\epsilon_1$,  $\underline{u}$ is a subsolution by the above lemma. By B. Guan \cite{guan1994cpde} (see also \cite{guan2014duke}), equation \eqref{case1EquaAppr} has a strictly $k$-convex solution $u^{\varepsilon,r}\in C^{\infty}(\overline \Omega_r)$. Our goal is to establish uniform $C^2$ estimates of $u^{\varepsilon,r}$, which are independent of $\varepsilon$ and $r$.
 
 We can check that $\bar u:=\Big(\frac{|x|}{r_0}\Big)^{2-\frac{n}{k}}$ is a supersolution of the above approximating equation. Indeed, $\bar u$ is smooth in $\Omega_r$ and $S_k ({D^2 \bar u})=0$. \\
 On $\p B_r$, we have $$u^{\varepsilon,r}=\frac{1}{2}\Big(\frac{r}{R_0}\Big)^{2-\frac{n}{k}}+\frac{r^2}{2R_0^2}\le\Big(\frac{r}{R_0}\Big)^{2-\frac{n}{k}} <\Big(\frac{r}{r_0}\Big)^{2-\frac{n}{k}}.$$
 On $\p\Omega$, since $B_{r_0}\subset\subset \Omega$, we have
 $$u^{\varepsilon,r}=1<\Big(\frac{|x|}{r}\Big)^{2-\frac{n}{k}}=\bar u,$$ 
 where we use $2k>n$. Thus by comparison principal, we have $u<\bar u$ in $\overline\Omega_r$.
 
Our goal is to prove the following estimates.
 \begin{theorem}\label{apu20720}
 	Assume $ k>\frac{n}{2}$. For  sufficiently small $\varepsilon$ and  $r$,    $u^{\varepsilon, r}$ satisfies
 	\begin{align*}
 		\left\{
 		\begin{aligned}
 			C^{-1}|x|^{\frac{2k-n}{k}}\le u^{\varepsilon, r}(x)\le&  C|x|^{\frac{2k-n}{k}},\\
 			|Du^{\varepsilon, r}|(x)\le& C|x|^{\frac{k-n}{k}},\\
 			|D^2u^{\varepsilon,r }|(x)\le& C|x|^{-\frac{n}{k}},
 		\end{aligned}
 		\right.
 	\end{align*}
 	where $C$ is a uniform constant  independent of $\varepsilon$ and $r$.
 \end{theorem}
\subsection{Case 2: $k<\frac{n}{2}$}
Since the Green function in this case is $-|x|^{\frac{2k-n}{k}}$, we want to solve the following $k$-Hessian equation .
\begin{equation}\label{case2Equa}
\left\{\begin{aligned}S_{k}(D^2 u)=0 \qquad&\text{in}\quad \Omega\setminus\{0\},\\
	u=-1\quad&\text{on}\quad \p\Omega\\
u=-|x|^{\frac{2k-n}{k}}+O(1) \ &\text{as}\ {x\rightarrow0}.
\end{aligned}
\right.
\end{equation}

Denote $w:=-
{{|x|}}^{2-\frac{n}{k}}+{R_0}^{2-\frac{n}{k}}-1+a_0\frac{|x|^2}{2R_0^2}$. We  choose  $a_0=\Big((1-\tau_0)^{2-\frac{n}{k}}-1\Big)R_0^{2-\frac{n}{k}}$ such that $w<-\frac{1}{2}\Big((1-\tau_0)^{2-\frac{n}{k}}-1\Big)R_0^{2-\frac{n}{k}}-1$ in $\overline\Omega$. {By the concavity of $S_k^\frac1k$,} we also have
\begin{equation*}
{S_k^\frac1k(D^2w)=S_k^\frac1k\big(D^2(-|x|^{2-\frac nk})+\frac{a_0}{2R_0^2}D^2|x|^2\big)\geq S_k^\frac1k(\frac{a_0}{R_0^2}I),}
\end{equation*}
then
\begin{align*}
	S_{k}(D^2w)
	\ge {C_n^k}a_0^kR_0^{-2k}.
\end{align*}
Then we construct a smoothly and strictly $k$-convex function $\underline {u}$ by Lemma \ref{Guan2002lemma} as follows. 
\begin{lemma}
	There exists a  strictly $k$-convex function $\underline {u}\in C^{\infty}{(\overline\Omega_r)}$ satisfying
	\begin{align}
		\underline{u}=&
		\left\{\begin{aligned}\label{0817:sub2}
			K_0\Phi^0-1
			\quad&\text{if}\ d(x)\le \frac{\mu_0}{M_0}, \\
			w\quad &\text{if}\ d(x)>\mu_0,
		\end{aligned}
		\right.\\
		\underline{u}\ge& \max\left\{ w, K_0\Phi^0-1\right\}\ \ \text{if}\ \frac{\mu_0}{M_0}\le d(x)\le \mu_0,\notag\\
		S_k(D^2 \underline{u})\ge&\epsilon_1:= \min\{{C_n^k}a_0^kR_0^{-2k}, {K_0^k}\epsilon_0\} \quad \text{in}\ \Omega_r,\notag
	\end{align}
where $K_0$ and $M_0$ are uniform constants.
\end{lemma}
\begin{proof}
	Applying Guan's Lemma \ref{Guan2002lemma} for $U=\Omega_{2\mu_0}$, $g=K_0\Phi^0-1$, $h=w$ and $\delta=\frac{1}{4}\Big((1-\tau_0)^{2-\frac{n}{k}}-1\Big)R_0^{2-\frac{n}{k}}$, we get a strictly and smoothly $k$-convex function $\underline{u}$ in $\Omega_{2\mu_0}$. In the following, we prove \eqref{0817:sub2}.
	
	For any $x\in \overline\Omega_{2\mu_0}\setminus\Omega_{\mu_0}:=\{x\in\overline\Omega:\mu_0\le d(x)\le {2\mu_0}\}$, 
	we have
	\begin{align}	
h-g=&-|x|^{2-\frac{n}{k}}+R_0^{2-\frac{n}{k}}+a_0\frac{|x|^2}{2R_0^2}-K_0\Phi^0\notag\\
\ge& -{r_0}^{2-\frac{n}{k}}+R_0^{2-\frac{n}{k}}+t_0^{-1} K_0 (1-e^{-t_0\mu_0})\notag\\
=&R_0^{2-\frac{n}{k}},
\end{align}
where we use  $K_0=\frac{t_0{r_0}^{2-\frac{n}{k}}}{1-e^{-t_0\mu_0}}$.
	This implies $\underline{u}=w$ in $\overline\Omega_{2\mu_0}\setminus\Omega_{\mu_0}$.
	
	For any $x\in\overline{\Omega_{\frac{\mu_0}{M_0}}}:=\{x\in\overline\Omega:d(x)\le \frac{\mu_0}{M_0}\}$, since $\Omega\subset\subset B_{(1-\tau_0)R_0}$, we have
	\begin{align}
		g-h=& |x|^{2-\frac{n}{k}}-R_0^{2-\frac{n}{k}}-a_0\frac{|x|^2}{2R_0^2}+K_0\Phi^0\notag\\
		\ge&\frac{1}{2}\Big((1-\tau_0)^{2-\frac{n}{k}}-1\Big)R_0^{2-\frac{n}{k}}+t_0^{-1}K_0(1-e^{-t_0\frac{\mu_0}{M_0}})\notag\\
		=&\frac{1}{4}\Big((1-\tau_0)^{2-\frac{n}{k}}-1\Big)R_0^{2-\frac{n}{k}}:=\delta
	\end{align}
	where $M_0$ is defined by $K_0(1-e^{-t_0\frac{\mu_0}{M_0}})=2t_0\delta$.
	This implies $\underline{u}=K_0\Phi^0+1$ in $\Omega_{\frac{\mu_0}{M_0}}$.\\
	At last, we define $\underline{u}=w$ in $\Omega_r\setminus\Omega_{2\mu_0}$. {In $\Omega_\frac{\mu_0}{M_0}$, by Lemma \ref{lem2.2}, $S_k(D^2\underline u)=S_k(K_0\Phi^0)\geq K_0^k\epsilon_0$. In $\Omega_r\backslash \Omega_{2\mu_0}$, $S_k(D^2\underline u)=S_k(D^2w)\geq C_n^ka_0^kR_0^{-2k}$. In $\Omega_{2\mu_0}\backslash \Omega_{\frac{\mu_0}{M_0}}$, by the concavity of $S_k^\frac1k$,  $S_k^\frac1k(D^2\underline u)\geq \frac{1+t(x)}{2}S_k^\frac1k(D^2w)+\frac{1-t(x)}2S_k^\frac1k(K_0D^2\Phi^0).$ } The proof is complete.
\end{proof}
We consider the approximating equation
\begin{equation}\label{case2EquaAppr}
\left\{\begin{aligned}S_{k}(D^2 u^{\varepsilon,r})= \varepsilon  \quad&\text{in}\quad \Omega_r,\\
u^{\varepsilon,r}= \underline{u}\quad &\text{on}\quad \partial\Omega_r.
\end{aligned}
\right.
\end{equation}
Then $\underline u$ is a strict subsolution of the above $k$-Hessian equation for any $\varepsilon$ small, by Guan \cite{guan1994cpde} (see also Guan \cite{guan2014duke}), equation \eqref{case2EquaAppr} has a strictly $k$-convex solution $u^{\varepsilon,r}\in C^{\infty}(\overline \Omega_r)$. By maximum principle and assmuing $r$ is sufficiently small, $u^{\varepsilon,r}< -1$ in $\Omega_r$.
We want to derive uniform $C^2$ estimates of $u^{\varepsilon,r}$, which are independent of $\varepsilon$ and $r$.
We prove the following
\begin{theorem}\label{apu10720}
Assume $1\le k<\frac{n}{2}$. For every sufficiently small $\varepsilon$ and  $r$,    $u^{\varepsilon, r}$ satisfies
\begin{align*}
\left\{
\begin{aligned}
C^{-1}|x|^{-\frac{n-2k}{k}}\le& -u^{\varepsilon, r}(x)\le C|x|^{-\frac{n-2k}{k}},\\
|Du^{\varepsilon, r}|(x)\le& C|x|^{-\frac{n-k}{k}},\\
|D^2u^{\varepsilon, r}|(x)\le& C|x|^{-\frac{n}{k}},
\end{aligned}
\right.
\end{align*}
where $C$ is a uniform constant independent of $\varepsilon$ and $r$.
\end{theorem}

\subsection{{Case 3:} $k=\frac{n}{2}$}
Since the Green function in this case is $\log |x|$, we want to solve the $k$-Hessian equation as follows
\begin{equation}\label{case3Equa}
\left\{\begin{aligned}S_{\frac{n}{2}}(D^2 u)=&0 \quad\text{in}\quad \mathring\Omega,\\
u=&0\ \ \ \text{on}\ \partial\Omega,\\
u(x)=&\log |x|+O(1) \ \text{as}\ {x\rightarrow 0}.
\end{aligned}
\right.
\end{equation}

\subsubsection{The approximating equation }

Denote $w:=\log\frac{|x|}{R_0}+a_0\frac{|x|^2}{2R_0^2}$ where $a_0=\frac{1}{2}\log\frac{1}{1-\tau_0}>0$ such that $w<\frac{1}{2}\log (1-\tau_0)$,
 {By the concavity of $S_k^\frac1k$,} we also have
\begin{equation*}
{S_\frac n2^\frac2n(D^2w)=S_k^\frac1k\big(D^2\log\frac{|x|}{R_0}+\frac{a_0}{2R_0^2}D^2|x|^2\big)\geq S_k^\frac1k(\frac{a_0}{R_0^2}I),}
\end{equation*}
then
\begin{align*}
	S_{\frac n2}(D^2w)
	\ge {C_n^\frac n2}a_0^{\frac n 2}R_0^{-n}.
\end{align*}
Then we construct a smoothly and strictly $k$-convex function $\underline {u}$ by Lemma \ref{Guan2002lemma} as follows. 
\begin{lemma}
	There exists a  strictly $k$-convex function $\underline {u}\in C^{\infty}{(\overline\Omega_r)}$ satisfying
	\begin{align}
		\underline{u}=&
		\left\{\begin{aligned}\label{0817:sub3}
			K_0\Phi^0
			\quad&\text{if}\ d(x)\le \frac{\mu_0}{M_0}, \\
			w\quad &\text{if}\ d(x)>\mu_0,
		\end{aligned}
		\right.\\
		\underline{u}\ge& \max\left\{ w, K_0\Phi^0\right\}\ \ \text{if}\ \frac{\mu_0}{M_0}\le d(x)\le \mu_0,\notag\\
		S_k(D^2 \underline{u})\ge&\epsilon_1:= \min\{{C_n^\frac n2}a_0^{\frac n 2}R_0^{-n}, {K_0^\frac n2}\epsilon_0\} \quad \text{in}\ \Omega,\notag
	\end{align}
	where $K_0$ and $M_0$ are uniform constants.
\end{lemma}
\begin{proof}
	Applying Guan's Lemma \ref{Guan2002lemma} for $U=\Omega_{2\mu_0}$, $g=K_0\Phi^0$, $h=w$ and $\delta=\frac{1}{4}\log\frac{1}{1-\tau_0}>0$, we get a strictly and smoothly $k$-convex function $\underline{u}$ in $\Omega_{2\mu_0}$. In the following, we prove \eqref{0817:sub3}.\\
For any $x\in \overline\Omega_{2\mu_0}\setminus\Omega_{\mu_0}:=\{x\in\overline\Omega:\mu_0\le d(x)\le {2\mu_0}\}$, 
we have
\begin{align}	
	h-g=&\log\frac{|x|}{R_0}+a_0\frac{|x|^2}{2R_0^2}-K_0\Phi^0\notag\\
	\ge& \log\frac{r_0}{R_0}+t_0^{-1}K_0(1-e^{-\mu_0t_0})\notag\\
	=& \log\frac{R_0}{r_0}
	\ge\log\frac{1}{1-\tau_0}\notag\\
	>&\delta,
\end{align}
where we choose $K_0=\frac{2t_0\log\frac{R_0}{r_0}}{1-e^{-\mu_0t_0}}$ and we use $r_0<(1-\tau_0)R_0$.

For any $x\in\overline{\Omega_{\frac{\mu_0}{M_0}}}:=\{x\in\overline\Omega:d(x)\le \frac{\mu_0}{M}\}$, since $\Omega\subset\subset B_{(1-\tau_0)R_0}$, we have
\begin{align}
	g-h=&K_0\Phi^0-w\notag\\
	\ge& t_0^{-1}K_0(e^{-\frac{\mu_0t_0}{M_0}}-1)-\log\frac{|x|}{R_0}-a_0\frac{|x|^2}{2R_0^2}\notag\\
	\ge&t_0^{-1}K_0(e^{-\frac{\mu_0t_0}{M_0}}-1)-\log(1-\tau_0)-\frac{a_0}{2}\notag\\
	=&2\delta>\delta,
	\end{align}
where we use $a_0=2\delta=\frac{1}{2}\log\frac{1}{1-\tau_0}>0$ and $M_0$ is determined by $K_0(1-e^{-\frac{\mu_0t_0}{M_0}})=t_0\delta$.

We finish the proof by defining $\underline{u}=w$ in $\Omega_r\setminus\Omega_{2\mu_0}$.
	\end{proof}
Then
we consider the following approximating equation 
\begin{align}\label{case3EquaAppr}
	\left\{ {\begin{aligned}
			S_k(D^2 u^{\varepsilon,r})=\varepsilon  \quad  &\text{in}\ \  \Omega_r,\\
			u=\underline{u} \quad &\text{on}\ \   \partial \Omega_r .  \\
	\end{aligned}} \right.
\end{align}

%

We will prove the following pointwise estimates
\begin{theorem}\label{apu30720}
	Assume $ k=\frac{n}{2}$.  For every sufficiently small $\varepsilon$ and $r$, for any $x\in\overline\Omega_r$ $u^{\varepsilon, r}$ satisfies
	\begin{align*}
		\left\{
		\begin{aligned}
			 |u^{\varepsilon, r}(x)-\log|x||\le& C,\\
		|Du^{\varepsilon, r}|(x)\le& C|x|^{-1},\\
			|D^2u^{\varepsilon, r}|(x)\le& C|x|^{-2},
		\end{aligned}
		\right.
	\end{align*}
	where $C$ is a uniform constant which is independent of $\varepsilon$ and $r$.
	\end{theorem}
In the next subsections, we will prove uniform $C^{2}$ estimates of solutions of equations \eqref{case1EquaAppr}, \eqref{case2EquaAppr} and \eqref{case3EquaAppr}. The key point is that these estimates are independent of $\varepsilon$ and $r$.
\subsection{$C^0$ estimates}
We first prove $u^{\varepsilon,r}$ is increasing with $r$. For any $r\ge \tilde r$, we have $u^{\varepsilon,\tilde r}\ge\underline{u}$ in  $\Omega_{\tilde r}$ and then
\begin{align*}
	\left\{\begin{aligned}
		S_k(D^2u^{\varepsilon,r})=\varepsilon=S_k(D^2u^{\varepsilon,\widetilde{r}})
		\quad &\text{in}\ \Omega_r,\\
		u^{\varepsilon,r}=u^{\varepsilon,\widetilde r}\quad &\text{on}\  \partial \Omega,\\
		u^{\varepsilon,r}=\underline{u}\le u^{\varepsilon,\widetilde r}\quad  &\text{on} \ \partial B_r.
	\end{aligned}
	\right.
\end{align*}
Applying the maximum principle in $\Omega_r$, we have
\begin{align}\label{C0Esti1}
	u^{\varepsilon,r}\le {u}^{\varepsilon,\tilde r}\ \text{in}\ \overline\Omega_r.
\end{align}
\begin{proposition}
		Let $u^{\varepsilon,r}$ be the $k$-convex solution of the approximating equation \eqref{case1EquaAppr}, \eqref{case2EquaAppr} or \eqref{case3EquaAppr}.
	For sufficiently small $\varepsilon$ and $r$,
	 for any $x\in \overline\Omega_r$, we have
	\begin{align*}
		\frac 12	R_0^{\frac{n}{k}-2}|x|^{2-\frac{n}{k}}\le u^{\varepsilon,r}(x)\le r_0^{\frac{n}{k}-2}|x|^{2-\frac{n}{k}} \quad &\text{if}\ k>\frac{n}{2},\\
		|x|^{-\frac{n-2k}{k}}-r_0^{\frac{n-2k}{k}}+1\le -u^{\varepsilon,r}(x)\le |x|^{-\frac{n-2k}{k}}-R_0^{\frac{n-2k}{k}}+1 \quad &\text{if}\ k<\frac{n}{2},\\
		\log|x|-\log R_0 \le u^{\varepsilon,r}\le \log|x|-\log r_0\quad &\text{if}\ k=\frac{n}{2}.
	\end{align*}
\end{proposition}
\begin{proof}
	The lower bound of $u^{\varepsilon, r}$ holds since   $u^{\varepsilon, r}\ge \underline{u}$.
	
	{\textbf{Case 1:} $k>\frac{n}{2}$}
	
	 We can check that $\bar u:=\Big(\frac{|x|}{r_0}\Big)^{2-\frac{n}{k}}$ is a supersolution of the above approximating equation. Indeed, $\bar u$ is smooth in $\Omega_r$ and $S_k ({D^2 \bar u})=0$. \\
	On $\p B_r$, we have $$u^{\varepsilon,r}=\frac{1}{2}\Big(\frac{r}{R_0}\Big)^{2-\frac{n}{k}}+\frac{r^2}{2R_0^2}\le\Big(\frac{r}{R_0}\Big)^{2-\frac{n}{k}} <\Big(\frac{r}{r_0}\Big)^{2-\frac{n}{k}}.$$
	On $\p\Omega$, since $B_{r_0}\subset\subset \Omega$, we have
	$$u^{\varepsilon,r}=1<\Big(\frac{|x|}{r}\Big)^{2-\frac{k}{n}}=\bar u,$$ 
	where we use $2k>n$. Thus
	we have
	\begin{align*}
		\left\{\begin{aligned}
			S_k(D^2u^{\varepsilon,r})=\varepsilon
			>0=S_k\left(D^2\Big({ r_0^{ \frac{n-2k}{k} }{ |x|^{-\frac{n-2k}{k}} }}\Big)\right)\quad &\text{in}\ \Omega_{{r}}\\
			u^{\varepsilon,r}=1<{r_0^{\frac{n-2k}{k}}}{|x|^{-\frac{n-2k}{k}}} \quad &\text{on}\  \partial \Omega \ ,\\
			u^{\varepsilon, r} =\underline{u}<
			{r_0^{\frac{n-2k}{k}}}{{|x|} ^{-\frac{n-2k}{k}}}\quad  &\text{on} \ \partial B_{ {r}}.
		\end{aligned}
		\right.
	\end{align*}
	By maximum principal,
	we have $u^{\varepsilon,r}\le \bar u$ in $\overline\Omega_r$.

{\textbf{Case 2:} $k<\frac{n}{2}$}

One can check $\bar u=-|x|^{2-\frac{n}{k}}+r_0^{2-\frac{n}{k}}-1$ is  a supersolution. Indeed,
 we have
\begin{align*}
\left\{\begin{aligned}
S_k(D^2u^{\varepsilon,r})=\varepsilon
>0=S_k\left(D^2\Big({ -{ |x|^{2-\frac{n}{k}} }}\Big)\right)\quad &\text{in}\ \Omega_{{r}}\\
u^{\varepsilon,r}=-1<\bar u \quad &\text{on}\  \partial \Omega \ ,\\
 u^{\varepsilon, r} =\underline{u}<
 \bar u\quad  &\text{on} \ \partial B_{ {r}}.
\end{aligned}
\right.
\end{align*}
Applying the maximum principle in $\Omega_{ r}$, we have
\begin{align}\label{C0Esti2}
u^{\varepsilon,r}\le \bar u.
\end{align}


\textbf{Case 3:} $k=\frac{n}{2}$

Since ${u^{\varepsilon,r}}=\underline{u}\le \log|x|-\log r_0$ on $\p B_r$, $\underline{u}=0< \log|x|-\log r_0$ on $\p\Omega$ and $S_k(D^2 u^{\varepsilon,r})=\varepsilon>0=S_{k}(D^2(\log|x|))$, then we have $u^{\varepsilon,r}\le \log|x|-\log r_0$.
\end{proof}
\subsection{Gradient estimates}
In this subsection, we prove the global gradient estimate based on our key estimate in \cite{MaZhang2022arxiv}. If we further assume $\Omega$ is starshaped, we can prove the positive lower bound of the gradient  and thus the level set of the approximating solution is compact.

Motivated by B. Guan \cite{gb2007imrn} where he proved the gradient estimate for the complex Monge-Ampere equation, we proved the following gradient estimate for the $k$-Hessian equation in \cite{MaZhang2022arxiv}.
\begin{theorem}
{Let $U\subset \mathbb R^n$ be a domain,} $u{\in C^3(U)\cap C^1(\overline U)}$ be {a} solution of the $k$-Hessian equation $S_k(D^2 u)=f$ {in $U$} and $u<0$ if $k\le\frac{n}{2}$ and $u>0$ if $k> \frac{n}{2}$.    Denote by
\begin{align}\label{threecase}
P=
\left\{ {\begin{array}{*{20}c}
   {|Du|^2e^{2u}, \quad \quad  \ k=\frac{n}{2}  }, \\
   {|Du|^2 u^{\frac{2(n-k)}{2k-n}} , \ \quad\  k>\frac{n}{2} }, \\
   {|Du|^2 (-u)^{-\frac{2(n-k)}{n-2k}}, \ k<\frac{n}{2} }.  \\
\end{array}} \right.
\end{align}
then we have the following gradient estimate
\begin{align}\label{Gradientestimate}
\max\limits_{U} P\le
\left\{ {\begin{array}{*{20}c}
		{
			\max \left\{\max\limits_{U} (e^{2u} |D\log f|^2)
			, \max\limits_{\p U}P\right\},\quad \ k=\frac{n}{2}  }, \\
		{
			\max \left\{\left(\frac{2k-n}{k(n+1-k)}\right)^2\max\limits_{U} (u^{\frac{2k}{2k-n}} |D\log f|^2),
			 \max\limits_{\p U}P \right\},\quad \ k>\frac{n}{2}  },  \\
   {
   \max\left\{\left(\frac{n-2k}{k(n+1-k)}\right)^2\max\limits_{U} \left((-u)^{-\frac{2k}{n-2k}} |D\log f|^2\right)
   , \max\limits_{\p U}P \right\},\quad \ k<\frac{n}{2}  },
\end{array}} \right.
\end{align}
\end{theorem}
Applying the above estimate in our setting i.e. we take $U=\Omega_r$ and $f=\varepsilon$, we get the following
\begin{proposition}
		Let $u^{\varepsilon,r}{\in C^3(\Omega_r)\cap C^1(\overline \Omega_r)}$ be {a} $k$-convex solution of the approximating equation \eqref{case1EquaAppr}, \eqref{case2EquaAppr} or \eqref{case3EquaAppr}.
	For sufficiently small $\varepsilon$ and $r$, we have
	\begin{align}
		\max_{\overline\Omega_r} P\le 	\max_{\p\Omega_r} P .
		\end{align}
	\end{proposition}

\begin{proposition}
	Let $u^{\varepsilon,r}{\in C^3(\Omega_r)\cap C^1(\overline \Omega_r)}$ be {a} $k$-convex solution of the approximating equation \eqref{case1EquaAppr}, \eqref{case2EquaAppr} or \eqref{case3EquaAppr}.
	For sufficiently small $\varepsilon$ and $r$, we have
	\begin{align}
		\max_{\overline\Omega_r} P\le 	 C .
	\end{align}
\end{proposition}

\begin{proof}
We only need to prove {{{ boundary gradient estimates}}}.

\textbf{For simplicity, we use $u$ instead of $u^{\varepsilon, r}$ during the proof.}

We will construct upper barriers near $\partial \Omega$ and $\partial B_r$ respectively.

\emph{\textbf{Case 1: $k>\frac{n}{2}$}}

Let  $h\in C^{\infty}(\overline\Omega_{r_0})$ be the unique solution of
\begin{align*}
	\left\{\begin{aligned}
		\Delta h=0\quad &\text{in}\  \Omega_{r_0},\\
		h=1 \quad &\text{on}\ \partial \Omega,\\
		h=\frac 1 2
		\quad &\text{on}\ \partial B_{r_1},
	\end{aligned}
	\right.
\end{align*}
where $r_1=2^{-\frac{k}{2k-n}}r_0$.
By maximum principle and the $C^0$ estimate of $u$, $\underline{u}\le u\le h$ in $\overline\Omega_{r_0}$. Then for any $x\in\partial \Omega$
\begin{align*}
	0<c_0\le h_{\nu}\le {u}_{\nu}(x)\le \underline{u}_{\nu}(x)\le C,
\end{align*}
where  $\nu$ is the outward normal  of $\partial \Omega$. Then
\begin{align}\label{0714gradient1.1}
	0<c\le \max\limits_{\partial\Omega}|Du|=
	\max\limits_{\partial\Omega}(u_{\nu})\le  C.
\end{align}
This proves that $P$ is uniformly bounded on $\partial \Omega$.

Next we show $P$ is uniformly bounded on $\partial B_r$.
We consider  $\tilde u(y):=r^{\frac{n}{k}-2}u(x)$ and $\tilde {\underline{u}}(y):=r^{\frac{n}{k}-2}\underline{u}(x)$ for $y:=\frac{x}{r}\in B_2\setminus B_1$. $\tilde u$ satisfies
\begin{equation}
	\left\{
	\begin{aligned}
		S_k(D^2\tilde u)=r^n\varepsilon\quad&\text{in}\ B_2\setminus \bar{B_1},\\
		\tilde u=\tilde w\quad&\text{on}\ \p B_1,
	\end{aligned}
	\right.
\end{equation}
where $\tilde w(y)=r^{\frac n k-2}w(x)$ and recall $\underline u=w$ in $\Omega_r\subset\Omega_r$.

By the $C^0$ estimate of $u$, we have 
$${R_0}^{\frac{n}{k}-2}|y|^{2-\frac n k}\le \tilde u\le {r_0}^{\frac{n}{k}-2}|y|^{2-\frac n k}.$$
Then $\tilde u$ is uniformly bounded in $\overline{B_2\setminus B_1}$.
Let $\tilde h(y)$ be the smooth function solving 
\begin{equation}
	\left\{
	\begin{aligned}
		\Delta \tilde h=0\quad &\text{in}\  B_{2}\setminus B_{1},\\
		\tilde h=\tilde w=\frac{R_0^{\frac n k-2}}{2}+\frac{a_0}{2R_0^2}r^{\frac{n}{k}} \quad &\text{on}\ \p B_1,\\
		\tilde h={r_0}^{\frac n k-2}2^{\frac{n}{k}-2}
		\quad &\text{on}\ \p B_{2}.
	\end{aligned}
	\right.
\end{equation}
Then $\tilde h$  is uniformly $C^2$ bounded in $\overline{B_2\setminus B_1}$.
By maximum principal, we have
\begin{align}
	\tilde{w}\le \tilde u\le \tilde h.
\end{align}
Then 
\begin{align}
	\tilde{w}_{\nu}\le \tilde u_{\nu}\le {\tilde h}_{\nu}\le C \ \text{on}\ \p B_1.
\end{align}
Note that on $\p B_1$, we have $$\tilde{w}_{\nu}= r^{\frac n k-1}w_{x_i}y_i>(1-\frac{n}{2k})R_0^{2-\frac{n}{k}} >0. $$
where we use $k>\frac{n}{2}$.
Thus we have 
$$ c\le |D\tilde u|\le C\ \text{on}\ \p B_1.$$
Therefore, we get
$$ c|x|^{1-\frac{n}{k}}\le |D u|\le C|x|^{1-\frac{n}{k}}\ \text{on}\ \p B_r.$$
This implies $P$ is uniformly bounded on $\partial B_r$.

In conclusion, when $k>\frac{n}{2}$, $P$ is uniformly bounded in $\overline\Omega_r$.

\emph{\textbf{Case2: $k<\frac{n}{2}$}}\\
The gradient estimate on $\p\Omega$ is similar as case 1. We only  prove the gradient estimate on $\p B_1$.
 We consider  $\tilde u(y):=r^{\frac{n}{k}-2}u(x)$ and $\tilde {w}(y):=r^{\frac{n}{k}-2}w(x)$ for $y:=\frac{x}{r}\in B_2\setminus B_1$. $\tilde u$ satisfies
\begin{equation}
	\left\{
	\begin{aligned}
		S_k(D^2\tilde u)=r^n\varepsilon\quad&\text{in}\ B_2\setminus \bar{B_1},\\
		\tilde u=\tilde w \quad&\text{on}\ \p B_1.
		\end{aligned}
	\right.
	\end{equation}
 By the $C^0$ estimate of $u$ and assuming $r$ is small enough, we have 
$$\frac12|y|^{2-\frac n k}\le -\tilde u\le 2|y|^{2-\frac n k}.$$
Then $\tilde u$ is uniformly bounded in $\overline{B_2\setminus B_1}$.
Let $\tilde h(y)$ be the smooth function solving 
\begin{equation}
	\left\{
	\begin{aligned}
		\Delta \tilde h=0\quad &\text{in}\  B_{2}\setminus B_{1},\\
		\tilde h=\tilde w \quad &\text{on}\ \p B_1,\\
		\tilde h=-\frac{1}{2}2^{2-\frac{n}{k}}
		\quad &\text{on}\ \p B_{2}.
			\end{aligned}
	\right.
	\end{equation}
Then $\tilde h$  is uniformly $C^2$ bounded in $\overline{B_2\setminus B_1}$.
By maximum principal, we have
\begin{align}
 \tilde{w}\le \tilde u\le \tilde h.
	\end{align}
Then 
\begin{align}
 \tilde{w}_{\nu}\le \tilde u_{\nu}\le \tilde h,
\end{align}
where  $\nu(y)=y$ is the outward normal to $\p B_1$. Note that  $$\tilde{w}_{\nu}=\frac{n}{k}-2+\frac{a_0}{R_0^2}r^{\frac{n}{k}}>\frac{n}{k}-2>0 \ \text{on}\ \p B_1,$$
where we choose $r$ small enough and use $k<\frac{n}{2}$.
Thus we have 
\begin{align} c\le |D\tilde u|\le C\ \text{on}\ \p B_1.\end{align}
Therefore, we get
$$ c|x|^{1-\frac{n}{k}}\le |D u|\le C|x|^{1-\frac{n}{k}}\ \text{on}\ \p B_r.$$
Thus $P$ is uniformly bounded on $\partial B_r$.

 In conclusion, when $k<\frac{n}{2}$, $P$ is uniformly bounded in $\overline\Omega_r$.

  \emph{\textbf{Case 3: $k=\frac{n}{2}$}}\\
  The gradient estimate on $\p\Omega$ is similar as case 1.  We will prove the gradient estimate on $\p B_1$. Define $\tilde u(y)=u(x)$ with $y=\frac{x}{r}\in\overline B_2\setminus B_1$, we have
 \begin{equation}
  	\left\{
  	\begin{aligned}
  		S_k(D^2\tilde u)=r^{n}\varepsilon\quad&\text{in}\ B_2\setminus \bar{B_1},\\
  		\tilde u=\log r \quad&\text{on}\ \p B_1.
  	\end{aligned}
  	\right.
  \end{equation}
  By the $C^0$ estimate of $u$:
  $$\log|y|-\log R_0\le \tilde u-\log r\le \log|y|-\log r_0.$$
  
  Let $\bar h(y)$ be the smooth function solving 
  \begin{equation}
  	\left\{
  	\begin{aligned}
  		\Delta \tilde h=0\quad &\text{in}\  B_{2}\setminus B_{1},\\
  		\tilde h=\tilde w \quad &\text{on}\ \p B_1,\\
  		\tilde h=\log r+\log\frac{2}{r_0}
  		\quad &\text{on}\ \p B_{2}.
  	\end{aligned}
  	\right.
  \end{equation}
 We have $|D\tilde h|\le C$ in $\overline{B_2\setminus B_1}$. By comparison, we have $\tilde w\le \tilde u\le \tilde h$. Recall $\tilde w= \tilde u= \tilde h$ on $\p B_1$, we get
 \begin{align}
 0<c\le \tilde w_{\nu}	\le \tilde u_{\nu}\le \tilde h_{\nu}\le C\quad\text{on} \  \p B_1,
 	\end{align}
 where $c$ and $C$ are uniform positive constants.
 Then we have
 \begin{align}
 	c\le |D\tilde u|=\tilde u_{\nu}\le C \quad\text{on} \  \p B_1
 	\end{align}
 Therefore 
  \begin{align}
 	cr^{-1}\le |D u|=\tilde u_{\nu}\le C r^{-1}\quad\text{on} \  \p B_r
 \end{align}
Thus $P$ is uniformly bounded on $\Omega_r$.
 \end{proof}
 \subsubsection{\textbf{Positive lower bound of $|Du|$ when $\Omega$ is strictly $(k-1)$ convex and starshaped.}}

 \begin{lemma}
 	Let $\Omega$ be strictly $(k-1)$ convex and starshaped.
 	Let $u$ be the $k$-convex solution of the approximating equation \eqref{case1EquaAppr}, \eqref{case2EquaAppr} or \eqref{case3EquaAppr}.
 For sufficiently small $\varepsilon$ and $r$, there exists a uniform constant $c_0$ such that for any $x\in \overline\Omega_r$
 	\begin{align}\label{0714gl1}
 		x\cdot Du(x)\ge
 		&c_0 |x|^{2-\frac{n}{k}}.
 	\end{align}
 	In particular,
 	\begin{align}\label{0715gl1}
 		| Du(x)|\ge c_0|x|^{1-\frac{n}{k}}.
 	\end{align}
 \end{lemma}

 \begin{proof}
 	Recall $F^{ij}=\frac{\partial }{\partial u_{ij}}( \log S_k(D^2 u))$.
 	By Maclaurin inequality, we have \begin{align}\label{0817:lg1}\mathcal{F}=(n-k+1)\frac{S_{k-1}}{S_{k}}\ge k(C_n^k)^{\frac{1}{k}}S_k^{-\frac 1k}\ge k\varepsilon^{-\frac 1 k}.
 		\end{align}
 	
 	We first prove the positive lower bound of $x\cdot Du(x)$ on $\p\Omega_r$.
 	In fact,  since $|Du|\ge c$ on $\p\Omega$ and $\Omega$ is starshaped, we have
 	\begin{align}
 		x\cdot Du=x\cdot\nu |Du|\ge c\min_{\p\Omega}	x\cdot\nu:=c_{1}>0.
 	\end{align} 
 	On $\p B_r$, since $Du=|Du|\nu=|Du|\frac{x}{r}$, we have
 	\begin{align}
 		x\cdot Du=r|Du|\ge cr^{2-\frac{n}{k}}.
 	\end{align}
 	
 	Then for any $x\in\p\Omega_r$, we have 
 	\begin{align}\label{0815:g1}
 		x\cdot Du\ge c_0|x|^{2-\frac{n}{k}}.
 	\end{align}
 	
 	\emph{Case 1: $k< \frac{n}{2}$}
 	
 	We consider the function $H:=x\cdot Du(x)-b_{11} u-b_{12}\frac{|x|^2}{2}$ with $b_{11}=\frac{c_0}{2}r_0^{2-\frac n k}$ and $b_{12}=\frac{c_0}{4R_0^{\frac n k}}$.
 	
 Since	$u\le r_0^{\frac{n}{k}-2}|x|^{2-\frac{n}{k}}$, by \eqref{0815:g1}, we have 
 $$H\ge \frac{c_0}{2}|x|^{2-\frac n k}-b_{11}u+ \frac{c_0}{2}|x|^{2-\frac n k}-b_{12}\frac{|x|^2}{2}>0\quad\text{on} \ \p\Omega_r.$$ 
 	On the other hand, we have
 	\begin{align}
 		F^{ij}H_{ij}=(2-b_{11})k-b_{12}\mathcal{F}\notag\\
 		\le   2k-b_{12}k\varepsilon^{-\frac{1}{k}}<0,	\end{align}
 	assume $\varepsilon\in (0,\varepsilon_0)$ with $\varepsilon_0\le \Big(\frac{b_{12}}{2}\Big)^{k}$.
 	 
 	 By maximum principle, $$H\ge \min_{\partial \Omega_r} H>0.$$ 
 	 
 	\emph{Case 2: $k< \frac{n}{2}$.}
 	Consider the function $H:=x\cdot Du(x)+b_{21}a_1 u-\frac{b_{12}}{2}|x|^2$. Our goal is to show $H$ is positive in $\overline\Omega_r$. Indeed,
 	By \eqref{0815:g1} and
 	$ -u\le C|x|^{2-\frac{n}{k}}$, for $b_{21}:=\frac{1}{2}C^{-1}a_0$ and $b_{22}=\frac{c_0}{2R_0^{\frac{n}{k}}}$, for any $x\in \p\Omega_r$, we have
 	\begin{align}\label{0815:g21}
 	H\ge& \frac{1}{2}x\cdot Du-\frac{b_{22}}{2}|x|^2\notag\\
 	\ge& \frac{1}{2}|x|^{2-\frac{n}{k}}(c_0-b_{22}R_0^{\frac{n}{k}})\notag\\
 	\ge&\frac{c_0}{4}|x|^{2-\frac{n}{k}}>0\quad\text{on}\ \p\Omega_r.
 	\end{align}
 On the other hand, we have
 	\begin{align}
 		F^{ij}H_{ij}=(2+b_{11}a_1)k-b_{12}\mathcal{F}\notag\\
 		\le   (2+b_{21})k-b_{22}k\varepsilon^{-\frac{1}{k}}<0,	\end{align}
where we use \eqref{0817:lg1} and we assume $\varepsilon\in (0,\varepsilon_0)$ with $\varepsilon_0\le \Big(\frac{b_{22}}{2(1+b_{21})}\Big)^{k}$ . 
By maximum principle, $$H\ge \min_{\partial \Omega_r} H>0.$$ 
%
In conclusion, we prove $H>0$ in $\overline \Omega_r$ and thus  \eqref{0714gl1} is obtained.

%

By maximum principle, we have
$H>\min_{\partial \Omega_r} H>0$. 

\emph{Case 3: $k=\frac{n}{2}$}

We consider $H=x\cdot Du(x)-b_{31}-b_{32}\frac{|x|^2}{2}$ which is positive on the boundary of $\overline \Omega_r$ if we take $b_{31}$ and $b_{32}$ small enough. Since $F^{ij}H_{ij}\le k\varepsilon^{-\frac{1}{k}}(\varepsilon^{\frac{1}{k}}-b_{32})<0$ for $\varepsilon$ small enough, we have $H=x\cdot Du(x)-b_{31}-b_{32}\frac{|x|^2}{2}>0$ in $\overline \Omega_r$ and we can get the desired estimate.

%
 	\end{proof}

\subsection{Second order estimates}
By the uniform gradient estimate, we have proved that $P$ is uniformly bounded in $\overline\Omega_r$. 
We will prove the second order estimate of the approximating equations based on the following second order estimate in \cite{MaZhang2022arxiv} by the second author and the third author.

\subsubsection{\textbf{The global second order estimate can be reduced to the boundary second order estimate} }

\begin{theorem}
Let $u\in C^4(\Omega_r)\cap C^2(\overline\Omega_r)$ be a $k$-convex solution of \eqref{case1EquaAppr} or \eqref{case2EquaAppr} or \eqref{case3EquaAppr}. Define 
$
 G=u_{\xi\xi}\varphi(P)h(u)
$,
then we have
\begin{align}\label{Secondorderestimate}
\max\limits_{\Omega_r}  G\le C+\max\limits_{\partial\Omega_r}  G.
\end{align}
where  
$h$ is defined by 
\begin{align}
	h(u)=
	\left\{ {\begin{array}{*{20}c}
			{ u^{\frac{n}{2k-n}}, \ \ \quad k>\frac{n}{2} }, \notag\\
			{(-u)^{-\frac{n}{n-2k}}, \ k<\frac{n}{2} }. \notag \\
			{e^{2u}, \quad  \quad\ k=\frac{n}{2}  }, \notag\\
	\end{array}} \right.
\end{align} and $\varphi$ is defined by
\begin{align}
	\varphi(t)=\left\{
	\begin{aligned}
		(M-t)^{-\tau}, k<n,\\
		1,           \qquad k=n,
		\end{aligned}
	\right.
	\end{align}
where $M:=2\max P+1$, $\tau$ is a uniform positive constant
\end{theorem}

\subsubsection{\textbf{Second order estimate on the boundary $\partial \Omega_r$}}.
The second order estimate on $\p\Omega$ is the same as \cite{CNSIII} (see also \cite{MaZhang2022arxiv}).
Here we only need prove the second order estimate on $\p B_r$.

\emph{\textbf{Tangential second derivatives estimates}}\\


For any $x_0\in \p B_r$, we choose the coordinate such that $x_0=(0,\cdots,
 0, r)$, then near $x_0$, $\p B_r$ is locally represented by $x_n=(r^2-|x'|^2)^{\frac{1}{2}}$ and $\frac{\p ^2 x_n}{\p x_{\alpha}\p x_{\beta}}(x_0)=r^{-1}\delta_{\alpha\beta}$ with $1\le \alpha,\beta\le n-1$.\\
 Since $u|_{\partial B_r}=constant$, we have
 \begin{align}\label{4627141}
 u_{\alpha\beta}(x_0)=&-u_n(x_0)\frac{\p ^2 x_n}{\p x_{\alpha}\p x_{\beta}}(x_0)=r^{-1}u_n(x_0)\delta_{\alpha\beta}\notag\\
 =&r^{-1} u_{\nu}(x_0)\delta_{\alpha\beta}.
 \end{align}

 Since we have the boundary gradient estimate on $\partial B_r$,
 \begin{align*}
Cr^{-\frac{n-k}{k}}\ge u_{\nu}(x)\ge cr^{-\frac{n-k}{k}},
\end{align*}
then by \eqref{4627141}, we have
\begin{align}
|u_{\alpha\beta}(x_0)|\le& Cr^{-\frac{n}{k}}\\
\{u_{\alpha\beta}(x_0)\}\ge& c r^{-\frac{n}{k}} \{\delta_{\alpha\beta}\}\label{0823:1}.
\end{align}

\emph{\textbf{ { Tangential-normal derivative estimates $\partial \Omega_r$}}}

 For any $x_0\in \partial B_r$,
choose the coordinate such that $x_0=(0,\cdots,0,r)$, $\partial B_r\cap B_{\frac{1}{2} r}(x_0)$ is represented by
\begin{align*}
	x_n=\rho(x')=(r^2-|x'|^2)^{\frac{1}{2}},
\end{align*}
Consider the tangential operator $T_{\alpha}=\left(x_\alpha \partial_{n}-x_n \partial_{\alpha}\right)$,$1\le \alpha\le n-1$. Since $u(x',\rho(x'))$ is constant, we have
\begin{align*}
	0=&	u_{\alpha}+u_{n}\rho_{\alpha}
	=u_{\alpha}-x_{\alpha}\rho^{-1}u_n
\end{align*}
Then {on} $\partial B_r\cap B_{\frac{r}{2}}(x_0)$, we have
\begin{align*}
	T_{\alpha} u=x_\alpha u_{n}-\rho u_{\alpha}=0.
\end{align*}

We consider the function
\begin{align*}
w=	A	_1(1-r^{-1}x_n)\pm r^{\frac{n-2k}{k}}T_{\alpha } u \ \text{in}\ B_r\cap B_{\frac{r}{2}}(x_0),
\end{align*}
where $A_1$ is positive large constant.
Since $x_0=(0\cdots,0, r)$ and $T_{\alpha}u=0$ on $\partial B_r$, we have $w(x_0)=0$. 
Since $T_\alpha u=0$ on ${\partial B_{r}\cap B_{\frac r 2}(x_0)}$, we have $w|_{\p B_{r}\cap B_{\frac r2}(x_0)}\ge 0$. \\
Since on $B_r\cap B_{\frac{r}{2}(x_0)}$, $r^{\frac{n-2k}{k}}|T_{\alpha} u|\le C_1r^{\frac{n-2k}{k}}|x||Du|\le C$ and $x_n\le\frac{7r}{8}$,  choosing $A_1>16C$, we have \begin{align}
	w\ge \frac18 A_1-C>C>0 \quad \text{on}\quad {B_1\cap   \partial B_{\frac{1}{2}r}(y_0)}.
	\end{align}

 Observe that $F^{ij}w_{ij}=\pm F^{ij}T_{\alpha}u=T_{\alpha}(F^{ij}u_{ij})=0$.
	By maximum principle, $w$ attains its minimum $0$ at $x_0$. Then we have
	\begin{align*}
		0\ge w_n(x_0)=-A_1r^{-1}\pm r^{\frac{n-k}{k}} u_{\alpha n}(x_0).
	\end{align*}
Then
$|u_{\alpha n}(x_0)|\le A_1 r^{-1}|Du(x_0)|\le Cr^{-\frac nk}$ and thus we have the uniform tangential-normal derivative estimates on $\partial B_R$.

\emph{\textbf{Double normal derivative estimates $\partial \Omega_r$}}\\
We can choose the coordinate at $x_0$ such that $u_{n}(x_0)=|Du|$ and $\{u_{\alpha\beta}(x_0)\}_{1\le \alpha,\beta\le n-1}$ is diagonal.

For any $x_0\in \partial B_r$, by \eqref{0823:1}, we have
\begin{align*}
	u_{nn}c_0r^{-\frac{n(k-1)}{k}}\le u_{nn}(x_0)S_{k-1}(u_{\alpha\beta}(x_0))=&S_k(D^2 u(x_0))-S_{k}(u_{\alpha\beta}(x_0))
	+\sum_{i=1}^{n-1}u_{in}^2S_{k-2}(u_{\alpha\beta})\\
	\le& \varepsilon+Cr^{-n}\le 2C r^{-n}.
\end{align*}
This gives $u_{nn}\le C r^{-\frac{n}{k}}$. On the other hand,  $u_{nn}\ge -\sum\limits_{i=1}^{n-1}u_{ii}\ge -cr^{-\frac{n}{k}}$.
Then we have $|u_{nn}(x_0)|\le Cr^{-\frac{n}{k}}$.



In conclusion, we obtain $|D^2 u(x)|\le C|x|^{-\frac{n}{k}}$ on the boundary $\partial\Omega_r $ and thus $|D^2u|(x)\le C|x|^{-\frac nk}$ for any $x\in \overline\Omega_r$.

\section{Proof of Theorem \ref{main07201}, Theorem \ref{main07202} and Theorem \ref{main07203}}
\subsection{Uniqueness}
The uniqueness follows from the comparison principle for $k$-convex solutions of the $k$-Hessian equation in bounded domains by Wang-Trudinger\cite{trudingerwang1997tmna} (see also \cite{trudinger1997cpde, urbas1990indiana}). See \cite{MaZhang2022arxiv} for the detailed argument.

\subsection{Existence and $C^{1,1}$-estimates}
The existence follows from the uniform $C^2$-estimates for $u^{\varepsilon, r}$.

For any fixed sufficiently small $\varepsilon>0$ and compact subset $K\subset\subset\Omega\setminus\{0\}$,
 there exist $r_0$ sufficiently small such that $K\subset\subset \Omega_r$,  $|u^{\varepsilon, r}|_{C^2({\Omega_{r_0}})}\le C(\epsilon, K)$ for any $r<r_0$.
 By Evans-Krylov theory,  $|u^{\varepsilon, r}|_{C^{2,\alpha}{(K)}}\le C(\epsilon, K,m)$. Then there exists a subsequence $u^{\varepsilon, r_i}$ converging in $C^{2,\beta}$-norm ($\beta<\alpha$) to a strictly $k$-convex $u^\varepsilon$ in $K$ and $u^{\varepsilon}\in C^{2,\alpha}(K)$ satisfies
 \begin{align}
 	\left\{\begin{aligned}
 		S_k(D^2 u^\varepsilon)=\varepsilon\qquad\text{in}\ \ \Omega\setminus\{0\},\\
 		u^\varepsilon=1 \qquad\text{on}\ \ \p\Omega.
 		\end{aligned}
 	\right.
 	\end{align}

Moreover, by Theorem \ref{apu10720}, we have the following estimate
 \begin{align*}
 	\left\{
 	\begin{aligned}
 		C^{-1}|x|^{-\frac{n-2k}{k}}\le& -u^{\varepsilon}(x)\le C|x|^{-\frac{n-2k}{k}},\\
 		|Du^{\varepsilon}|(x)\le& C|x|^{-\frac{n-k}{k}},\\
 		|D^2u^{\varepsilon}|(x)\le& C|x|^{-\frac{n}{k}},
 	\end{aligned}
 	\right.
 \end{align*}
 Thus there exits a subsequence $u^{\epsilon_i}$ converges to $u$ in $C^{1,\alpha}_{loc}$ such that $u\in C^{1,1}(\overline\Omega\setminus\{0\})$ is the $k$-convex solution of the k-Hessian equation \eqref{case1Equa1.1} and satisfies the estimates \eqref{decay10720}.

 \textbf{Case 2: $k<\frac{n}{2}$}\\
 Similar as case 1,  there exists a subsequence $u^{\varepsilon, r_i}$ converging smoothly to a strictly $k$-convex $u^\varepsilon$ in $K$ and $u^{\varepsilon}\in C^{2,\alpha}(\Omega\setminus\{0\})$ satisfies
 \begin{align}
 	\left\{\begin{aligned}
 		S_k(D^2 u^\varepsilon)=\varepsilon\qquad\text{in}\ \ \Omega\setminus\{0\},\\
 		u^\varepsilon=-1, \qquad\text{on}\ \ \p\Omega.
 	\end{aligned}
 	\right.
 \end{align}

 Moreover, by Theorem \ref{apu20720}, we get
 \begin{align*}
 	\left\{
 	\begin{aligned}
 		 |u^{\varepsilon}(x)-|x|^{\frac{2k-n}{k}}|\le& C,\\
 	|Du^{\varepsilon}|(x)\le& C|x|^{-\frac{n-k}{k}},\\
 		|D^2u^{\varepsilon}|(x)\le& C|x|^{-\frac{n}{k}},
 	\end{aligned}
 	\right.
 \end{align*}
 Thus there exits a subsequence $u^{\epsilon_i}$ converges to $u$ in $C^{1,\alpha}_{loc}$ such that $u\in C^{1,1}(\overline\Omega\setminus\{0\})$ is the $k$-convex solution of the $k$-Hessian equation \eqref{case2Equa1.2} and satisfies the estimates \eqref{decay20720}.

 \textbf{Case 3: $k=\frac{n}{2}$}\\
Similar as case 1,  there exists a subsequence $u^{\varepsilon, r_i}$ converging smoothly to a strictly $k$-convex $u^\varepsilon$ in $K$ and $u^{\varepsilon}\in C^{\infty}(\Omega\setminus\{0\})$ satisfies
 \begin{align}
 	\left\{\begin{aligned}
 		S_k(D^2 u^\varepsilon)=\varepsilon\qquad\text{in}\ \ \Omega\setminus\{0\},\\
 		u^\varepsilon=0, \qquad\text{on}\ \ \p\Omega.
 	\end{aligned}
 	\right.
 \end{align}

 Moreover, by Theorem \ref{apu30720}, we get
 \begin{align*}
 	\left\{
 	\begin{aligned}
 		|u^{\varepsilon}(x)-\log|x||\le& C,\\
 	|Du^{\varepsilon}|(x)\le& C|x|^{-1},\\
 		|D^2u^{\varepsilon}|(x)\le& C|x|^{-2},
 	\end{aligned}
 	\right.
 \end{align*}
 Thus there exits a subsequence $u^{\epsilon_i}$ converges to $u$ in $C^{1,\alpha}_{loc}$ such that $u\in C^{1,1}(\Omega\setminus\{0\})$ is the $k$-convex solution of the $k$-Hessian equation \eqref{case3Equa1.3} and satisfies the estimates \eqref{decay30720}.

\section{A monotonicity formula along the level set of the approximating solution}
 Agostiniani-Mazzieri \cite{AM2020CVPDE} proved an monotonicity formula along the level set of the solution of the following problem
 \begin{align}
 	\left\{
 	\begin{aligned}
 		\Delta u=0 \ \text{in} \ \Omega^c\\
 		u=-1  \ \text{on} \ \p\Omega\\
 		\lim\limits_{|x|\rightarrow\infty}u(x) =0.
 		\end{aligned}
 	\right.
 	\end{align}
 Since the solution of the homogeneous $k$-Hessian equation is only $C^{1,1}$, we consider the level set of $u^{\varepsilon}$. 
In \cite{MaZhang2022arxiv}, we prove an monotonicity formula along the level set of the solution of the exterior Dirichlet problem of the approximating $k$-Hessian equation. As an application of our uniform $C^{1,1}$ estimates of $u^{\varepsilon}$ and the positive lower bound of $|Du^{\varepsilon}|$, we prove an interior version of \cite{MaZhang2022arxiv}. 

 We firstly estimate the area of the level set $S_{t}=\{x\in\Omega\setminus\{0\}: u^{\varepsilon}(x)=t\}$.

 \begin{lemma}\label{st0723}
There exits uniform constant $C$ such that
 	\begin{align}
 		|S_t|	\le&
 		\left\{\begin{aligned} Ct^{\frac{k(n-1)}{2k-n}} \quad &{\forall}	 t\in (0,1]\quad \text{if} \ k>\frac{n}{2},\\
 			C|t|^{-\frac{k(n-1)}{n-2k}} \quad &{\forall}\ 	 t\in (-\infty,-1]\quad \text{if} \ k<\frac{n}{2},\\
 			 Ce^{(n-1)t}\quad &{\forall} t\in (-\infty, 0]\quad \text{if} \ k=\frac{n}{2}.
 		\end{aligned}
 		\right.
 		\end{align}
 	\end{lemma}
 	 \begin{proof}
 	 	For any fixed $t$, assume $r>0$ sufficiently small,
 	  we have
	\begin{align*}
	  {|S_t|-|\partial B_r|=\int_{\{u<t\}\setminus{B_r}}\mathrm{div}\Big(\frac{Du^{\varepsilon}}{|Du^{\varepsilon}|}\Big)dx}. 
 \end{align*}
\textbf{Case1: $k>\frac{n}{2}$}

For any $x\in\{x:u(x)<t\}$, since $|D^2u^{\varepsilon}|(x)\le C|x|^{-\frac{n}{k}}$ and $|Du^{\varepsilon}|\ge c|x|^{1-\frac{n}{k}}$, we have 
\begin{align*}
	\Big|\mathrm{div}\Big(\frac{D u^{\varepsilon}}{|Du^{\varepsilon}|}\Big)\Big|=\Big|\frac{\Delta u^{\varepsilon}}{|Du^{\varepsilon}|}-\frac{u^{\varepsilon}_{ij}u^{\varepsilon}_iu^{\varepsilon}_j}{|Du^{\varepsilon}|^3}\Big|\le C|D^2u^{\varepsilon}||Du^{\varepsilon}|^{-1}\le C|x|^{-1}.
	\end{align*}
Combining the above estimate with  $\{u<t\}\subset B_{Ct^{\frac{k}{2k-n}}}$, we have
\begin{align}
0\le |S_t|-|\partial B_r|\le& 	 C\int_{B_{Ct^{\frac{k}{2k-n}}}}{|x|^{-1}}dx\notag\\
\le& C\int_{0}^{Ct^{\frac{k}{2k-n}}} s^{n-2} ds\notag\\
\le&\ Ct^{(n-1)\frac{k}{2k-n}}.
	\end{align}
 Taking $r\rightarrow 0$, we have 
\begin{align*}
	|S_t|\le\ C|t|^{(n-1)\frac{k}{2k-n}}.
	\end{align*}

\textbf{Case2: $k<\frac{n}{2}$.}
Similar argument shows that
\begin{align}
	 |S_t|-|\partial B_r|\le& C\int_{0}^{C|t|^{-\frac{k}{2n-k}}} s^{n-2}ds\notag\\
	\le&\ C|t|^{-(n-1)\frac{k}{n-2k}} .
\end{align}
\textbf{Case3: $k=\frac{n}{2}$.}

We have
\begin{align}
 |S_t|-|\partial B_r|\le&\  C\int_{0}^{C\mathrm{e}^t} s^{n-2}ds\notag\\
 \le&  \ C{e}^{(n-1)t}.
\end{align}

 \end{proof}

Similar as the exterior case in \cite{MaZhang2022arxiv}, we consider the following quantity
\begin{align}
	I_{a, b, k}(t):=
	\int_{S_t}
	g^a(u^{\varepsilon})|Du^{\varepsilon}|^{b-k}S_k^{ij}(D^2u^{\varepsilon})u^{\varepsilon}_iu^{\varepsilon}_j,
\end{align}
where $g(u^{\varepsilon})$ is defined by
\vspace*{-5pt}
\begin{align}
	g(u^{\varepsilon})=\left\{
	\begin{aligned}
		&(u^{\varepsilon})^{\frac{n-k}{2k-n}}, \ \quad k>\frac{n}{2},\\
		&(-u^{\varepsilon})^{\frac{n-k}{2k-n}}, k<\frac{n}{2},\\
	&{e}^{u^{\varepsilon}},\ \qquad k=\frac{n}{2}.
		\end{aligned}
	\right.
	\end{align}
We choose $a=b-k+1$ and one can see that $I_{a,b,k}(t)$ is uniformly bounded  due to the $C^2$ estimates of $u^{\varepsilon}$ and the positive lower bound of $|Du^{\varepsilon}|$.
We define
\begin{align}	J_{a+a_0,b,k}(t,t_0):=g^{a_0}(t) I'_{a,b,k}(t)-g^{a_0}(t_0)  I'_{a,b,k}(t_0).
	\end{align}
We prove the following useful equality along the level set of $u^{\varepsilon}$.
\begin{proposition}\label{0725mono1} Let $u^{\varepsilon}$ be the solution of the approximating k-Hessian equation with $a=b-k+1$.
	We have the following identity
		\begin{align}\label{monotone20210831:1}
		J_{a+a_0,b,k}(t,t_0)
		=& -ba\int_{t_0}^{t}\int_{S_{s}}
	\Big(g^{a+a_0}|Du^{\varepsilon}|^{b-k-1}\frac{H_k}{H_{k-1}}S_k\Big)dAds+(b+1)\int_{t_0}^{t}\int_{S_s}\Big(g^{a+a_0-1}g'|Du^{\varepsilon}|^{b-k}S_k\Big)dAds\notag\\&
	+(b+1)\int_{S_t}\Big(g^{a+a_0}|Du^{\varepsilon}|^{b-k}S_k\Big)dA-(b+1)\int_{S_{t_0}}\Big(g^{a+a_0}|Du^{\varepsilon}|^{b-k}S_k\Big)dA\notag\\
		&+a\int_{t_0}^{t}\int_{S_{s}}g^{a+a_0}|Du^{\varepsilon}|^{b-1}H_{k-1}^{-1}\Big(c_{n,k}H_{k}^2-(k+1)H_{k-1}H_{k+1}\Big)dAds\notag\\
		&+a\int_{t_0}^{t}\int_{S_{s}}
		g^{a+a_0}
		|Du^{\varepsilon}|^{b-1}\mathcal{L}\ dA ds-ab\int_{t_0}^{t}\int_{S_s}{g^{a+a_0}|Du^{\varepsilon}|^{b-k-2}\mathcal{M}} \ dAds,
		\end{align}
where $H_m$ is the $m$-th order fundamental symmetric function of principal curvatures $m$-Hessian operator of the level set $S_s$ of $u^{\varepsilon}$, $a_0, b, c_{n,k}=\frac{k(n-k-1)}{n-k}$ and the functions $\mathcal{L}$ are chosen as follows
\begin{enumerate}[{(i)} ]
		\item	If $1\le k<\frac{n}{2}$, we require  $-\infty< t_0<t\le -1$, $a_0=-2\frac{n-2k}{n-k}$ and
		$\mathcal{L}=(b-c_{n,k})\Big(
		\frac{n-k}{n-2k}|D\log u^{\varepsilon}|-\frac{H_k}{H_{k-1}}\Big)^2.$
	\end{enumerate}
\begin{enumerate}[{(ii)} ]
	\item	If $k=\frac{n}{2}$, we require  $-\infty< t<t_0\le 0$, $a_0=0$ and  $\mathcal{L}=a\Big(
	|D u^{\varepsilon}|-\frac{H_k}{H_{k-1}}\Big)^2$.
\end{enumerate}
\begin{enumerate}[{(iii)} ]
	\item	If $k>\frac{n}{2}$, we require  $0< t<t_0\le 1$, $a_0=2\frac{2k-n}{n-k}$, $\mathcal{L}=(b-c_{n,k})\Big(
	\frac{n-k}{n-2k}|D\log u^{\varepsilon}|-\frac{H_k}{H_{k-1}}\Big)^2$.
\end{enumerate}

 and \begin{align}\label{m}
 	\mathcal{M}:=S_{k+1}&-\frac{H_k}{H_{k-1}}|Du^{\varepsilon}|S_k+\frac{H_k^2}{H_{k-1}}|Du^{\varepsilon}|^{k+1}-H_{k+1}|Du|^{k+1}\le 0.
 	\end{align}

	\end{proposition}
\begin{proof}
		\textbf{For simplicity, we use $u$ instead of $u^{\varepsilon}$ and $S_k$ intead of $S_k(D^2u^{\varepsilon})$ during the proof.}
		
		We use the notation $\Omega_t:=\{x\in\Omega\setminus\{0\}: u(x)>t\}$ and we define $\Omega_{t_0t}:=\Omega_{t_0}\setminus\overline\Omega_{t}$ for any $t_0<t$.
		
	By  the divergence theorem and the divergence free property of the $k$-Hessian operator i.e. $\sum\limits_{j=1}^nD_jS_k^{ij}=0$, we have
	\begin{align}
		I_{a,b,k}(t)-I_{a,b,k}(t_0)=&\int\limits_{\Omega_{t_0t}}D_j\left(g^a|Du|^{b+1-k}S_{k}^{ij}u_i\right)\notag\\
		=&a\int_{\Omega_{t_0t}}g^{a-1}g'|Du|^{b+1-k}S_k^{ij}u_iu_j\notag\\
		&+(b+1-k)\int_{\Omega_{t_0t}}g^{a}|Du|^{b-k-1}S_k^{ij}u_iu_lu_{ij}+k\int_{\Omega_{t_0t}}g^a|Du|^{b+1-k}S_k\notag\\
		=&a\int_{\Omega_{t_0t}}g^{a-1}g'|Du|^{b+1-k}S_k^{ij}u_iu_j\notag\\
		&-(b+1-k)\int_{\Omega_{t_0t}}g^{a}|Du|^{b-k-1}S_{k+1}^{ij}u_iu_{j}+(b+1)\int_{\Omega_{t_0t}}g^a|Du|^{b+1-k}S_k\notag\\
		=&a\int_{t_0}^{t}\int_{S_{s}}g^{a-1}g'|Du|^{b-k}S_k^{ij}u_iu_j-(b-k+1)\int_{t_0}^{t}\int_{S_{s}}g^a|Du|^{b-k-2}S_{k+1}^{ij}u_iu_{j}\notag\\
				&+(b+1)\int_{t_0}^{t}\int_{S_s}g^{a}|Du|^{b-k}S_k,\label{0726It}
	\end{align}

where we use $S_k^{ij}u_iu_lu_{lj}=|Du|^2S_k-S_{k+1}^{ij}u_iu_{j}$ and the coarea formula.

Then
	\begin{equation}\label{0726I}
		\begin{aligned}
			I'_{a,b,k}(t)=&a\int_{S_t}g^{a-1}g'|Du|^{b-k}S_k^{ij}u_iu_j\\
			&-(b+1-k)\int_{S_t}g^{a}|Du|^{b-k-2}S_{k+1}^{ij}u_iu_{j}+E_{a,b,k}(t),
		\end{aligned}
	\end{equation}
where $E_{a,b,k}(t)=(b+1)\int_{S_t}g^{a}|Du|^{b-k}S_k$.

Then we have
\begin{align}\label{10724}
	&J_{a+a_0,b,k}(t,t_0)=g^{a_0}(t) I'_{a,b,k}(t)-g^{a_0}(t_0) I'_{a,b,k}(t_0)\notag\\
	=&a\int_{\Omega_{t_0t}}D_j\Big(g^{a+a_0-1}g'|Du|^{b-k+1}S_k^{ij}u_i\Big)\notag\\
	&-a\Big(I_{a+a_0,b-1,k+1}(t)-I_{a+a_0,b-1,k+1}(t_0)\Big)+E_{a+a_0,b,k}(t)-E_{a+a_0,b,k}(t_0),
	\end{align}
where we use $a=b-k+1$.
We will compute the terms in \eqref{10724}.\\
Firstly we have
\begin{align}\label{30724}	
&\int_{\Omega_{t_0t}}D_j\Big(g^{a+a_0-1}g'|Du|^{b-k+1}S_k^{ij}u_i\Big)dx\notag\\
=&	\int_{t_0}^{t}\int_{S_s}\Big((g^{a+a_0-1}g'\Big)'|Du|^{b-k}S_k^{ij}u_iu_j dAds
\notag\\
&+(b-k+1)\int_{t_0}^{t}\int_{S_s}\Big(g^{a+a_0-1}g'|Du|^{b-k-2}S_k^{ij}u_iu_lu_{lj}\Big)dAds+k\int_{t_0}^{t}\int_{S_s}\Big(g^{a+a_0-1}g'|Du|^{b-k}S_k\Big)dAds\notag\\
=&\int_{t_0}^{t}\int_{S_s}\Big((g^{a+a_0-1}g')'|Du|^{b-k}S_k^{ij}u_iu_j\Big)dAds\notag\\
&-(b-k+1)\int_{t_0}^{t}\int_{S_s}\Big(g^{a+a_0-1}g'|Du|^{b-k-2}S_{k+1}^{ij}u_iu_j\Big)dAds+(b+1)\int_{t_0}^{t}\int_{S_s}\Big(g^{a+a_0-1}g'|Du|^{b-k}S_k\Big)dAds\notag\\
=&\int_{t_0}^{t}\int_{S_s}\Big((g^{a+a_0-1}g')'|Du|^{b+1}H_{k-1}\Big)dAds
-(b-k+1)\int_{t_0}^{t}\int_{S_s}\Big(g^{a+a_0-1}g'|Du|^{b}H_{k}\Big)dAds\notag\\
&+(b+1)\int_{t_0}^{t}\int_{S_s}\Big(g^{a+a_0-1}g'|Du|^{b-k}S_k\Big)dAds,
	\end{align}
where we use the identity $H_{m-1}|Du|^{m+1}={S_{m}^{ij}}{u_{i}u_{j}}$ for $m\in \{1,2,\cdots, n\}$ (see e.g. \cite{Reilly1973,Trudinger1997,Salani2008}).

For the term $I_{a+a_0,b-1,k+1}(t)-I_{a+a_0,b-1,k+1}(t_0)$, similar as the manipulation of \eqref{0726It}, we have
\begin{align}\label{20724}
	I_{a+a_0,b-1,k+1}(t_0)-&I_{a+a_0,b-1,k+1}(t)\notag\\
	=&(a+a_0)\int_{t_0}^{t}\int_{S_{s}}g^{a+a_0-1}g'|Du|^{b-k-2}S_{k+1}^{ij}u_iu_j dAds\notag\\
	&-(b-1-k)\int_{t_0}^{t}\int_{S_{s}}g^{a+a_0}|Du|^{b-k-4}S_{k+2}^{ij}u_iu_{j}\notag\\
	&+b\int_{t_0}^{t}\int_{S_s}g^{a+a_0}|Du|^{b-k-2}S_{k+1}.
	\end{align}

Next we deal with the above term involving $S_{k+1}$. Choose the coordinate such that   $u_n(x_0)=|Du|(x_0)$ and $\{u_{ij}(x_0))\}_{1\le i,j \le n-1}=\{\tilde\lambda_i\delta_{ij}\}_{1\le i,j \le n-1}$ is diagonal, we have
\begin{align*}
	S_{k+1}=&u_{nn}S_{k}(\tilde\lambda)+S_{k+1}(\tilde\lambda)-\sum\limits_{i=1}^{n-1}S_{k-1}(\tilde\lambda|i)u_{ni}^2\\
	S_{k}=&u_{nn}S_{k-1}(\tilde\lambda)+S_{k}(\tilde\lambda)-\sum\limits_{i=1}^{n-1}S_{k-2}(\tilde\lambda|i)u_{ni}^2,
	\end{align*}
where $\tilde\lambda=(\tilde\lambda_1,\cdots,\tilde\lambda_{n-1})$ and  recall we use the notation $S_k=S_k(D^2 u)$. Then we get
\begin{align}\label{unn0724}
	S_{k+1}=&\frac{S_k(\widetilde\lambda)}{S_{k-1}(\widetilde\lambda)}S_k-\frac{S_k^2(\widetilde\lambda)}{S_{k-1}(\tilde\lambda)}+\sum_{i=1}^{n-1}u_{ni}^2\frac{S_{k}(\tilde \lambda|i)S_{k-2}(\tilde \lambda|i)-S_{k-1}^2(\tilde\lambda|i)}{S_{k-1}(\tilde \lambda)}+S_{k+1}(\tilde\lambda).
	\end{align}
Noting that $S_{m}(\tilde \lambda)=|Du|^{-2}S_{m+1}^{ij}u_iu_j=H_{m}|Du|^{m}$ is  globally defined, we obtain $$S_{k+1}-\frac{H_k}{H_{k-1}}|Du^{\varepsilon}|S_k+\frac{H_k^2}{H_{k-1}}|Du^{\varepsilon}|^{k+1}-H_{k+1}|Du|^{k+1}=\sum_{i=1}^{n-1}u_{ni}^2\frac{S_{k}(\tilde \lambda|i)S_{k-2}(\tilde \lambda|i)-S_{k-1}^2(\tilde\lambda|i)}{S_{k-1}(\tilde \lambda)}\leq 0.$$ This proves \eqref{m}
Inserting \eqref{unn0724} into \eqref{20724} and noting that $S_{m}(\tilde \lambda)=|Du|^{-2}S_{m+1}^{ij}u_iu_j=H_{m}|Du|^{m}$ is  globally defined , then we have
\begin{align}\label{20724:1}
	I_{a+a_0,b-1,k+1}(t)&-I_{a+a_0,b-1,k+1}(t_0)\notag\\ =&(a+a_0)\int_{t_0}^{t}\int_{S_{s}}g^{a+a_0-1}g'|Du|^{b}H_k dAds\notag\\
	&+(k+1)\int_{t_0}^{t}\int_{S_{s}}g^{a+a_0}|Du|^{b-1}H_{k+1}dAds\notag\\
	&-b\int_{t_0}^{t}\int_{S_s}g^{a+a_0}|Du|^{b-1}\frac{H_k^2}{H_{k-1}}dAds
	+b\int_{t_0}^{t}\int_{S_s}g^{a+a_0}|Du|^{b-k-1}\frac{H_k}{H_{k-1}}S_kdAds\notag\\
	&+b\int_{t_0}^{t}\int_{S_s}g^{a+a_0}|Du|^{b-k-2}\left(S_{k+1}-\frac{H_k}{H_{k-1}}|Du|S_k+\frac{H_k^2}{H_{k-1}}|Du|^{k+1}-H_{k+1}|Du|^{k+1}\right)
\end{align}
Inserting \eqref{30724} and \eqref{20724} into \eqref{10724}, we obtain
\begin{align}
	J_{a+a_0,b,k}(t,t_0)=& -ba\int_{t_0}^{t}\int_{S_{s}}
	g^{a+a_0}|Du|^{b-k-1}\frac{H_k}{H_{k-1}}S_kdAds\notag\\
	&+(b+1)\int_{S_t}g^{a+a_0}|Du|^{b-k}S_kdA-(b+1)\int_{S_{t_0}}g^{a+a_0}|Du|^{b-k}S_kdA\notag\\
	&+a\int_{t_0}^{t}\int_{S_{s}}g^{a+a_0}|Du|^{b-1}H_{k-1}^{-1}\Big(c_{n,k}H_{k}^2-(k+1)H_{k-1}H_{k+1}\Big)dAds\notag\\
	&+a\int_{t_0}^{t}\int_{S_{s}}
	g^{a+a_0}
	|Du|^{b-1}H_{k-1} \mathcal{L}dAds\notag\\
	&-ab\int_{t_0}^{t}\int_{S_s}g^{a+a_0}|Du|^{b-k-2}\left(S_{k+1}-\frac{H_k}{H_{k-1}}|Du|S_k+\frac{H_k^2}{H_{k-1}}|Du|^{k+1}-H_{k+1}|Du|^{k+1}\right).
	\end{align}
where the function $\mathcal{L}$ is defined by
\begin{align}\label{0723}
\mathcal{L}=&(b-c_{n,k})\Big(
\frac{H_k}{H_{k-1}}\Big)^2-(2a+a_0) (\log g)'|Du|\frac{H_k}{H_{k-1}}\notag\\
&+\Big((\log g)''+(a+a_0)((\log g)')^2\Big)|D u|^2.
\end{align}
Now we divide two cases to prove the $\mathcal{L}\ge 0$ under some restrictions on $a$ and $b$.

\textbf{Case1: $k<\frac{n}{2}$ and $\frac{n}{2}<k<n$.}\\
We choose $c_{n,k}=\frac{k(n-k-1)}{n-k}$.
Then we have \begin{align}
	(\log g)''+(a+a_0)((\log g)')^2=&\frac{n-k}{n-2k}u^{-2}+(a+a_0)(\frac{n-k}{n-2k})^2u^{-2}\notag\\
=&(\frac{n-k}{n-2k})^2u^{-2}(\frac{n-2k}{n-k}+a+a_0)\notag\\
=&(b-c_{n,k})(\frac{n-k}{n-2k})^2u^{-2},
\end{align}
where we choose $a_0=-2\frac{n-2k}{n-k}$ and we use  $a=b-k+1$. We also have
\begin{align}
	-(2a+a_0)(\log g)'=2\frac{n-k}{n-2k}(b-c_{n,k})u^{-1}.
	\end{align}
 Then we have

\begin{align}
\mathcal{L}
	=&(b-c_{n,k})\Big(
	\frac{n-k}{n-2k}|D\log u|-\frac{H_k}{H_{k-1}}\Big)^2.
	\end{align}
Consequently,  we obtain the desired identity.

\textbf{Case 2: $k=\frac{n}{2}$.}\\
We have $c_{n,k}=\frac{n}{2}-1>0$. We require $b\ge \frac{n}{2}-1$,  $a=b-\frac{n}{2}+1=b-c_{n,k}\ge 0$ and $a_0=0$.\\
Since $g=e^{u}$ and thus $(a+a_0)^{-1}(g^{a+a_0})''=(a+a_0)g^{a+a_0}$. We obtain
\begin{align}
	\mathcal{L}=a\Big(
	|D u|-\frac{H_k}{H_{k-1}}\Big)^2.
	\end{align}
At last we prove $\mathcal{M}:=S_{k+1}-\frac{H_k}{H_{k-1}}|Du|S_k+\frac{H_k^2}{H_{k-1}}|Du|^{k+1}-H_{k+1}|Du|^{k+1}$ is non-positive similar as that in Ma-Zhang \cite{MaZhang2022arxiv}.


%
	\end{proof}
From the above formula, we have the following almost monotonicity formula along the level set of $u^{\varepsilon}$ and we prove the first part of Theorem \ref{geometric0725}.

 \begin{proposition}\label{0727G1}
 	Let $u^{\varepsilon}$ be the solution of the approximating k-Hessian equation.
 	Assume $\frac{n}{2}<k<n$
 	and $b\ge c_{n,k}=\frac{k(n-k-1)}{n-k}$ and $b\neq k-1$, then for any $t\in(0,1]$, we have
 	\begin{align}
 	\frac{d}{dt}  I_{a,b,k}(t)	\left\{
 		\begin{aligned}
 		 \ge -C\varepsilon  t^{\frac{nk}{2k-n}-1} &\ \text{if} \ a>0,\\
 		  \le C\varepsilon |t|^{\frac{nk}{2k-n}-1}&\ \text{if} \ a< 0.
 		\end{aligned}\right.
 	\end{align}
 	In particular, we have the following weighted inequality
 	\begin{align}\label{0723geometric}
 		\int_{\p\Omega}{|Du|^{b+1}H_{k-1}}\ge \frac{2k-n}{n-k}\int_{\p\Omega}{|Du|^{b}H_{k}},
 	\end{align}
 where $u$ is the unique $C^{1,1}$ solution of the homogeneous $k$-Hessian equation  \eqref{case1Equa1.1}.
 \end{proposition}
\begin{proof}

	We divide two cases.
	
	\textbf{Case1: $a> 0$}\\
By  Proposition \ref{0725mono1}, for any $0< t_0<t\le 1$,  we have
\begin{align}
	t^2  I'_{a,b,k}(t)-t_0^2I'_{a,b,k}(t_0)=&J_{a+a_0,b,k}(t)-J_{a+a_0,b,k}(t_0)\notag\\
	\ge&
	-ab\int_{\Omega_{t_0}\setminus\overline\Omega_t}(u^{\varepsilon})^{a\frac{n-k}{2k-n}+2}|Du^{\varepsilon}|^{a-1}\frac{H_k}{H_{k-1}}S_k\notag\\
	&-(b+1)  \int_{S_{t_0}}(u^{\varepsilon})^{a\frac{n-k}{2k-n}+2}|Du^{\varepsilon}|^{a-1}S_k.
	\end{align}
By the MacLaurin  inequality: $\frac{H_k}{H_{k-1}}\le \frac{C_{n-1}^{k}}{C_{n-1}^{k-1}}\Big(\frac{H_{k-1}}{C_{n-1}^{k-1}}\Big)^{\frac{1}{k-1}}$ and the uniform $C^2$-estimates of $u^\varepsilon$ (we also use  $|Du^{\varepsilon}|\ge c|x|^{1-\frac{n}{k}}$), for any  $x\in \Omega_t^{c}$, we have
\begin{align*}
	(u^{\varepsilon})^{a\frac{n-k}{2k-n}+2}|Du^{\varepsilon}|^{a-1}\frac{H_k}{H_{k-1}}\le& C(u^{\varepsilon})^{a\frac{n-k}{2k-n}+2}|Du^{\varepsilon}|^{a-1}H_{k-1}^{\frac{1}{k-1}}\notag\\
	\le& C|x|^{a\frac{n-k}{k}+2\frac{2k-n}{k}}|x|^{(a-1)\frac{k-n}{k}}|x|^{-1}\notag\\
	=&C|x|^{2-\frac{n}{k}}\le Ct,
	\end{align*}
then
 \begin{align*}
	\int_{\Omega_{t_0}\setminus\overline\Omega_t}(u^{\varepsilon})^{a\frac{n-k}{2k-n}+2}|Du^{\varepsilon}|^{a-1}\frac{H_k}{H_{k-1}}S_k\le C\varepsilon t^{\frac{n(k-1)+2k}{2k-n}}.
\end{align*}
Similarly, we have
\begin{align*}
	\int_{S_{t_0}}(u^{\varepsilon})^{a\frac{n-k}{2k-n}+2}|Du^{\varepsilon}|^{a-1}S_k\le C\varepsilon t_0^{\frac{n(k-1)+2k}{2k-n}},
	\end{align*}
where we use $|S_{t_0}|\le Ct_0^{^{\frac{k(n-1)}{2k-n}}}$ (see Lemma \ref{st0723}).

Thus we get
\begin{align}\label{110725}
	t^2I'_{a,b,k}(t)&-t_0^2 I'_{a,b,k}(t_0)
\ge-C\varepsilon t^{\frac{n(k-1)+2k}{2k-n}}-C\varepsilon t_0^{\frac{n(k-1)+2k}{2k-n}}.
	\end{align}
 By the uniform $C^{2}$ estimates for $u^{\varepsilon}$ and  $|Du^{\varepsilon}|\ge c|x|^{1-\frac{n}{k}}$,
we have for any $t_0\in(0,1]$
\begin{align}
	t_0^2 \Big|I'_{a,b,k}(t_0)\Big|\le Ct_0.
	\end{align}

Let $t_0$ tend to $0$ in \eqref{110725}, we have
\begin{align}
	  I'_{a,b,k}(t)\ge -C\varepsilon t^{\frac{nk}{2k-n}-1}.
	\end{align}
In particular,  taking $t=1$ we have
\begin{align}
 	 I'_{a,b,k}(1)\ge -C\varepsilon.
	\end{align}
On the other hand, by \eqref{0726I}, we have
\begin{align}
	 I'_{a,b,k}(1)
\le a\frac{n-k}{n-2k}\int_{\p\Omega}{|Du^{\varepsilon}|^{b+1}H_{k-1}}- a\int_{\p\Omega}{|Du^{\varepsilon}|^{b}H_{k}}+C\varepsilon.
\end{align}
Consequently, we get
\begin{align}\label{0723final}
\frac{n-k}{2k-n}\int_{\p\Omega}{|Du^\varepsilon|^{b+1}H_{k-1}}- \int_{\p\Omega}{|Du^\varepsilon|^{b}H_{k}}\ge -a^{-1}C\varepsilon
\end{align}
Since $|Du^{\varepsilon}|$ converges to $|Du|$ on $\p \Omega$, we finish the proof of \eqref{0723geometric} by taking $\varepsilon\rightarrow 0$ in \eqref{0723final}.

\textbf{Case2: $a<0$}

Similar as case 1, we have 
\begin{align}
	I'_{a,b,k}(t)\le C\varepsilon t^{\frac{nk}{2k-n}-1}.
	\end{align}
On the other hand, we have
\begin{align}
I'_{a,b,k}(t)\ge 	a\frac{n-k}{n-2k}\int_{\p\Omega}{|Du^{\varepsilon}|^{b+1}H_{k-1}}- a\int_{\p\Omega}{|Du^{\varepsilon}|^{b}H_{k}}.
	\end{align}
Then the desired inequality follows.
	\end{proof}

Next we prove the second part  of Theorem \ref{geometric0725}.
\begin{lemma}\label{0727G2}
	Assume $k=\frac{n}{2}$ and $b\ge\frac{n}{2}-1$. We have
	\begin{align}
	I'_{a, b, k}(t)\ge -C\varepsilon e^{nt},
	\end{align}
In particular, we have
\begin{align}
	\int_{\p\Omega}|Du|^{b+1}H_{k-1}
	\ge \int_{\p\Omega}|Du|^{b}H_{k},
\end{align}
where $u$ is the unique $C^{1,1}$ solution the homogeneous $k$-Hessian equation  \eqref{case3Equa1.3}.
	\end{lemma}
\begin{proof}
By Proposition \ref{0725mono1} and similar as the argument in Proposition \ref{0727G1}, for any $-\infty< t_0\le s<t\le 0$, we have
\begin{align*}
I'_{a,b,k}(t)-I'_{a,b,k}(s)\ge - C\varepsilon e^{nt}.
\end{align*}
Integrating the above from $t_0$ to $t$, we have
\begin{align*}
	I_{a,b,k}(t)-I_{a,b,k}(t_0)\le
	\Big(I'_{a,b,k}(t)+ C\varepsilon e^{nt}\Big) (t-t_0),
	\end{align*}
Then
\begin{align*}
	\Big(I'_{a,b,k}(t)+ C\varepsilon e^{nt}\Big)(-tt_0^{-1}+1)\ge -t_0^{-1}(I_{a,b,k}(t)-I_{a,b,k}(t_0))\ge Ct_0^{-1}.
\end{align*}
let $t_0$ tend to $0$ and note that  $I_{a,b,k}(t)$ is uniformly bounded which follows from the $C^2$-estimates of $u^{\varepsilon}$ and $|Du^{\varepsilon}|\ge c |x|^{1-\frac{n}{k}}$, 
we obtain
\begin{align*}
	I'_{a,b,k}(t)\ge -C\varepsilon e^{nt}.
	\end{align*}
On the other hand, we have\begin{align*}
	I'_{a,b,k}(0)\le& a\int_{\p\Omega}|Du^{\varepsilon}|^{b+1}H_{k-1}
	-a\int_{\p\Omega}|Du^{\varepsilon}|^{b}H_{k}+C\varepsilon.
\end{align*}
Combining the above two inequalities and noting that $|Du^{\varepsilon}|\rightarrow |Du|$,
we get
\begin{align}
	\int_{\p\Omega}|Du|^{b+1}H_{k-1}
	\ge \int_{\p\Omega}|Du|^{b}H_{k}.
	\end{align}
	\end{proof}

When $k<\frac{n}{2}$, we have the following inequality.

\begin{lemma}
	Let $u^{\varepsilon}$ be the solution of the approximating k-Hessian equation.
	Assume $k<\frac{n}{2}$,
	and $b\ge c_{n,k}$, then for any $-\infty< t_0\le t\le -1$, we have
	\begin{align}
		t^2I'_{a,b,k}(t)  -t_0^2I'_{a,b,k}(t_0)\ge -C\varepsilon|t|^{-\frac{nk}{n-2k}-1}.
	\end{align}
	
\end{lemma}

{\bf Acknowledgements:}
The second author was supported by  National Natural Science Foundation of China (grants 11721101 and 12141105) and National Key Research and Development Project (grants SQ2020YFA070080). 

\bibliographystyle{plain}
\bibliography{interiorkHessian-2023-3-12}

\end{document}